\documentclass[12pt, reqno]{amsart}

\numberwithin{equation}{section}
\usepackage{amsfonts, amssymb, latexsym, epsfig, amsthm, enumitem}
\usepackage{hyperref}
\usepackage[capitalise, noabbrev]{cleveref}
\usepackage[usenames,dvipsnames]{xcolor}
\usepackage{graphicx,pgf,tikz}
\usetikzlibrary{decorations.markings}
\usepackage{ytableau}
\usepackage[all]{xy}
\usepackage{wrapfig}
\usepackage{mathdots}
\usepackage{ bbold }
\usepackage{dirtytalk}
\usepackage{soul}

\newlist{thmlist}{enumerate}{1}
\setlist[thmlist]{label=(\roman{thmlisti}),noitemsep}
\usepackage[top=1in, bottom=1in, left=1in, right=1in]{geometry}

\usepackage{enumitem}

\newtheorem{theorem}{Theorem}[section]
\newtheorem{proposition}[theorem]{Proposition}
\newtheorem{lemma}[theorem]{Lemma}

\newtheorem*{claim*}{Claim}

\newtheorem{Main Conjecture}[theorem]{Main Conjecture}

\theoremstyle{definition}
\newtheorem{defn}[theorem]{Definition}
\newtheorem{examplex}[theorem]{Example}
\newenvironment{example}
  {\pushQED{\qed}\examplex}
  {\popQED\endexamplex}

\newtheorem{remark}[theorem]{Remark}
\newtheorem{notation}[theorem]{Notation}
\theoremstyle{plain}
\newtheorem{question}{Question}
\newtheorem{introthm}{Theorem}

\newcommand{\init}{{\tt in}}
\newcommand{\hgt}{{\rm{ht }}}

\newcommand{\ZZ}{\mathbb{Z}}

\newcommand{\PP}{\mathbb{P}}
\newcommand{\QQ}{\mathbb{Q}}
\newcommand{\A}{\mathcal{A}}
\newcommand{\B}{\mathcal{B}}
\newcommand{\G}{\mathcal{G}}
\renewcommand{\H}{\mathcal{H}}

\renewcommand{\P}{\mathcal{P}}

\newcommand{\bb}{\mathbf{b}}
\newcommand{\be}{\mathbf{e}}

\newcommand{\Tor}{{\sf Tor}}

\DeclareMathOperator{\pd}{pd}
\DeclareMathOperator{\lk}{lk}
\DeclareMathOperator{\del}{del}
\DeclareMathOperator{\reg}{reg}
\DeclareMathOperator{\depol}{depol}
\DeclareMathOperator{\HS}{H}
\DeclareMathOperator{\cone}{cone}

\newcommand{\qand}{\quad \mbox{and} \quad}
\newcommand{\qwith}{\quad \mbox{with} \quad}
\newcommand{\qforall}{\quad \mbox{for all} \quad}
\newcommand{\qforsome}{\quad \mbox{for some} \quad}
\newcommand{\qforeach}{\quad \mbox{for each} \quad}
\newcommand{\qwhere}{\quad \mbox{where} \quad}

\title[Polarization and Gorenstein Liaison]{Polarization and Gorenstein Liaison}
\author[Faridi]{Sara Faridi}
\address
{Department of Mathematics \& Statistics,
Dalhousie University,
Halifax, NS,
Canada 
}
\email{faridi@dal.ca}
\author[Klein]{Patricia Klein}
\address{Department of Mathematics, Texas A\&M University, College Station, TX, USA}
\email{pjklein@tamu.edu}
\author[Rajchgot]{Jenna Rajchgot}
\address{Department of Mathematics and Statistics, McMaster University, Hamilton, ON, Canada}
\email{rajchgoj@mcmaster.ca}
\author[Seceleanu]{Alexandra Seceleanu}
\address{Mathematics Department, University of Nebraska--Lincoln, Lincoln, NE, USA.}
\email{aseceleanu@unl.edu}

\thanks{Faridi was supported by NSERC Discovery Grant 2023-05929. Klein was partially supported by NSF DMS-2246962. Rajchgot was partially supported by NSERC Discovery Grants 2017-05732 and 2023-04800. Seceleanu was partially supported by NSF DMS–2101225 and DMS-2401482.}

\date{\today}
\begin{document}
\begin{abstract}
A major open question in the theory of Gorenstein liaison is whether or not every arithmetically Cohen--Macaulay subscheme of $\PP^n$ can be G-linked to a complete intersection.  Migliore and Nagel showed that, if such a scheme is generically Gorenstein (e.g., reduced), then, after re-embedding so that it is viewed as a subscheme of $\PP^{n+1}$, indeed it can be G-linked to a complete intersection. Motivated by this result, we consider techniques for constructing G-links on a scheme from G-links on a closely related reduced scheme. 

Polarization is a tool for producing a squarefree monomial ideal from an arbitrary monomial ideal.  Basic double G-links on squarefree monomial ideals can be induced from vertex decompositions of their Stanley--Reisner complexes.  Given a monomial ideal $I$ and a vertex decomposition of the Stanley--Reisner complex of its polarization $\P(I)$, we give conditions that allow for the lifting of an associated basic double G-link of $\P(I)$ to a basic double G-link of $I$ itself. We use the relationship we develop in the process to show that the Stanley--Reisner complexes of polarizations of stable Cohen--Macaulay monomial ideals are vertex decomposable.

We then introduce and study polarization of a Gr\"obner basis of an arbitrary homogeneous ideal and give a relationship between geometric vertex decomposition of a polarization and elementary G-biliaison that is analogous to our result on vertex decomposition and basic double G-linkage.
\end{abstract}
\maketitle

\setcounter{tocdepth}{1} 
\tableofcontents

\section{Introduction}
\label{sec:intro}

A broad goal of  liaison theory in its various forms is to study the properties of possibly complicated schemes by relating them to simpler ones, such as complete intersections.

Given a  subscheme  $X$ of $\PP^n$,  there are many questions about $X$ we consider fundamental: What is the codimension of $X$?  What is the degree of $X$?  Is $X$ arithmetically Cohen--Macaulay, and, if not, how close is it to being so?

When $X$ is an arbitrary subscheme, these questions are often quite difficult.  When $X$ is a complete intersection, however, they can become easy, at least if we already know its homogeneous defining ideal $I_X$.  In that case, the codimension of $X$ is the number of minimal generators of $I_X$, the degree of $X$ may be computed by B\'ezout's Theorem (over an algebraically closed field), and $X$ is necessarily arithmetically Cohen--Macaulay.  If we are interested in some other subscheme $X'$ of $\PP^n$, it is a great asset to be able to relate $X'$ to the complete intersection $X$ in a controlled enough manner so as to be able to infer information about $X'$ from information about $X$.  

The inferences are particularly strong in the classical setting of complete intersection liaison (or, CI-liaison), introduced by Peskine and Szpiro \cite{PS74} and studied extensively by many since, including notably Huneke and Ulrich \cite{HU87}, who work in a much more general algebraic setting than we consider here.  Suppose that $C_1$ and $C_2$ are subschemes of the complete intersection $X$ with homogeneous, saturated defining ideals $I_{C_1}$, $I_{C_2}$, and $I_X$, respectively.  We say that $C_1$ and $C_2$ are directly CI-linked by $X$ if $I_X:I_{C_1} = I_{C_2}$ and $I_X:I_{C_2} = I_{C_1}$.  In this case, the $C_i$ are equidimensional of the same dimension as $X$ and neither has an embedded component; $\deg(C_1)+\deg(C_2) = \deg(X)$; $X = C_1 \cup C_2$ if the $C_i$ share no component; the higher cohomology modules of $C_1$ and $C_2$ agree up to graded shift and dual; in particular, $C_1$ and $C_2$ are either both or neither arithmetically Cohen--Macaulay; resolutions of any two of $I_{C_1}$, $I_{C_2}$, and $I_X$ determine a resolution (including maps) of the third; and, if $C_1$ and $C_2$ are arithmetically Cohen--Macaulay, then they are either both or neither (strongly) unobstructed.  See \cite{HU87, KMM+01} for an overview of these and related results and also \cite{MN21} for a focus on applications.   One generates equivalence classes, called CI-liaison classes, from these direct links.  

CI-liaison is able to preserve so much structure because the equivalence classes are relatively small. Indeed, in each codimension $>2$, the arithmetically Cohen--Macaulay schemes are partitioned into infinitely many distinct CI-liaison classes \cite[Corollary 7.10]{KMM+01}. The small size of the CI-liaison classes means that, given a scheme $X$, it may be difficult to find another scheme $Y$ in the same CI-liaison class as $X$ such that $Y$ is meaningfully easier to study than $X$. Thus, in certain settings, it can be valuable to trade some of the structure that CI-liaison preserves for the added flexibility afforded by larger equivalence classes, or, otherwise said, by more allowable links.

One good candidate for this trade is Gorenstein liaison (or, G-liaison), introduced by Schenzel \cite{Sch83}. The definition of G-liaison is the same as that of CI-liaison with the sole modification that $X$ is permitted to be merely arithmetically Gorenstein, not necessarily a complete intersection. G-liaison classes are larger than CI-liaison classes, yet G-liaison still preserves many good properties, as we will soon discuss. A priori, another natural candidate might be to allow $X$ to be any arithmetically Cohen--Macaulay scheme.  However,  this is undesirable because, for that notion of liaison, all equidimensional schemes of the same codimension without embedded components are in the same liaison class \cite[Corollary 0.3]{Wal93}. In this paper, we will be focusing on G-liaison.

 If the subschemes $C_1$ and $C_2$ are directly G-linked by the arithmetically Gorenstein scheme $X$, then, as is the case with CI-liaison, still the $C_i$ are equidimensional of the same dimension as $X$ and neither has an embedded component; $\deg(C_1)+\deg(C_2) = \deg(X)$; $X = C_1 \cup C_2$ if the $C_i$ share no component; and $C_1$ and $C_2$ are either both or neither arithmetically Cohen--Macaulay.  Some statements about resolutions, cohomology, and unobstructedness, however, require restriction on the types of G-links allowed or additional assumptions on the $C_i$.  See again \cite{KMM+01} and also \cite{Nag98}, especially for properties preserved under an even number of direct G-links, a case which will be a main focus throughout this paper.  

Gorenstein liaison has found many applications. For example, it can be used to establish Gr\"obner bases, construct schemes with desired properties, and compute invariants (see, for example, \cite{MN03, CN09, DG11, GMN13, GN14, CNP+17, FGM18, FK20, KLL+20, CDF+, KW, Kle23, Ney23, NRV}). Many of these results owe to the ability of Gorenstein liaison to track Hilbert functions and to its tendency to be compatible with combinatorial constructions.  It is with these sorts of fundamentally graded applications in mind that we are motivated to study Gorenstein liaison in $\PP^n$, i.e., in the setting of standard graded polynomial rings with links that are required to be homogeneous.

The extent to which one might consider the sacrifice of the additional structure provided by CI-liaison to be worthwhile depends, to many, on how many more schemes can be linked to a complete intersection in the more lenient framework of Gorenstein liaison.  The major open question in this area is

\begin{question}\cite[Question 1.6]{KMM+01}\label{quest:ACM==glicci} 
Is every arithmetically Cohen--Macaulay subscheme of $\PP^n$ in the \underline{G}orenstein \underline{l}iaison \underline{c}lass of a \underline{c}omplete \underline{i}ntersection (abbreviated glicci)?
    \end{question}

\noindent Because all complete intersections of a fixed codimension are in the same G-liaison class, it is equivalent to ask if there is exactly one G-liaison class of each codimension containing any (or all) of the arithmetically Cohen--Macaulay schemes.

Large families of arithmetically Cohen--Macaulay subschemes are known to be glicci, including standard determinantal schemes \cite{KMM+01}, mixed ladder determinantal schemes from two-sided ladders \cite{Gor07}, schemes of Pfaffians \cite{DG09}, wide classes of arithmetically Cohen--Macaulay curves in $\PP^4$ \cite{CMR00, CMR01}, arithmetically Cohen--Macaulay schemes defined by Borel-fixed monomial ideals \cite{MN02}, arithmetically Gorenstein schemes \cite{CDH05}, and schemes defined by so-called geometrically vertex decomposable ideals \cite{KR21}.  

Quite powerfully, Gorla \cite{Gor08} obtained the broad result that every scheme defined by minors of a fixed size of a matrix with polynomial entries is glicci provided it has the expected codimension, generalizing results of \cite{KMM+01} and also \cite{Har07}.  This result, which shows that generalized determinantal schemes are often glicci, stands in sharp contract to examples in CI-liaison. 
 For example, even the scheme determined by the maximal minors of a $2 \times 4$ generic matrix fails to be in the CI-liaison class of a complete intersection, a direct consequence of \cite[Corollary 5.13]{HU87}, which shows that determinantal schemes are rarely in the CI-liaison class of a complete intersection.  

In \cite[Theorem 2.3]{MN13}, Migliore and Nagel showed that every generically Gorenstein (e.g., reduced) arithmetically Cohen--Macaulay subscheme of $\PP^n$ is glicci when viewed instead as a subscheme of $\PP^{n+1}$.  With Migliore and Nagel's result in mind, the primary object of study in this paper is the relationship between G-links of schemes and G-links of closely related reduced schemes.  Specifically, we study basic double G-links (\cref{def:basic double G-link}) and elementary G-biliaisons (\cref{def:Gbiliaison}), two related forms of G-liaison that generate the same equivalence classes. In many of the applications of Gorenstein liaison cited above, it is basic double G-linkage or G-biliaison that is implemented.  This history motivates our interest in these particular liaison constructions.

In the monomial case, we use the classical tool of polarization of a monomial ideal (\cref{def:polarization}) to relate monomial basic double G-links of ideals and monomial basic double G-links of their respective polarizations, which enjoy the advantages of the Stanley--Reisner correspondence. Stanley-Reisner theory has long been a cornerstone of combinatorial commutative algebra. 
The rich interplay between combinatorics and algebra that this theory unveils has led to significant advancements in understanding the algebraic structures underlying simplicial complexes. Stanley-Reisner theory interacts with liaison theory by means of special classes of simplicial complexes termed vertex decomposable (see \cref{def:vertex decomposable}). Nagel and R\"omer \cite{NR08} showed that the Stanley--Reisner scheme of a (weakly) vertex decomposable simplicial complex is G-linked to a complete intersection defined by indeterminates via a sequence of G-links that respects the vertex decomposition of the simplicial complex. However, the scope of Stanley-Reisner theory is inherently limited to the squarefree case, prompting a natural question: 

\begin{question}\label{q: nonsquarefree liaison}
Can one extend the relationship between monomial ideals and liaison beyond the setting of squarefree monomial ideals? 
\end{question}

In this paper we answer \cref{q: nonsquarefree liaison}  using the classical tool of polarization. Here and elsewhere, we say that two saturated, homogeneous ideals of a polynomial ring are G-linked if the schemes they define are G-linked. In \cref{s: pol and liaison} we develop methods for establishing a basic double G-link between two not necessarily squarefree monomial ideals based on knowledge of a basic double G-link between their respective polarizations.  

\begin{introthm}[\cref{cor:summaryMonomialCase}]
Let $I$ be a monomial ideal in a polynomial ring and $\P(I)$ its  polarization. The following are equivalent:
\begin{enumerate}
\item $I$ is glicci via a sequence of monomial basic double G-links of shift $1$;
\item $\P(I)$ is glicci via a sequence of monomial basic double G-links of shift $1$, and certain auxiliary ideals involved in the sequence define generically Gorenstein schemes;
\item $\P(I)$ is the Stanley-Reisner ideal of a weakly vertex decomposable simplicial complex, and certain auxiliary ideals arising during vertex decomposition define generically Gorenstein schemes.
\end{enumerate}
\end{introthm}

  Using these ideas, and resting on the relationship between vertex decomposable and Gorenstein liaison developed by Nagel and R\"omer \cite{NR08}, we show that polarizations of certain natural classes of monomial ideals are glicci.

\begin{introthm}[\cref{thm:StableVertexDecomp,thm:ArtinianVertexDecomp}]
     The polarizations of stable Cohen--Macaulay monomial ideals and  artinian monomial ideals are glicci, as their  Stanley--Reisner complexes are vertex decomposable.
 \end{introthm}

   \cref{thm:ArtinianVertexDecomp} recovers \cite[Remark 1.8]{M2011}, which states that the Stanley--Reisner complexes of polarizations of artinian monomial ideals are vertex decomposable. Motivation for these investigations comes from recent work on the topology of the Stanley-Reisner complexes of polarized ideals in \cite{AFL22, FM22} as well as from analogous versions of our results which hold for the non-polarized ideals cf. \cite{MN02}. 

   Interpreting a special case of Question 1 algebraically, there is no known example of a squarefree monomial ideal that is Cohen–Macaulay but not glicci.  Murai~\cite{M2011} showed that polarizations of Artinian monomial ideals are vertex decomposable, which we recover as Theorem 3.16.  Combining this result with ~\cite[Theorem 3.3]{NR08} reveals that some of the best known classes of Cohen--Macaulay squarefree monomial ideals studied in combinatorial commutative algebra algebra-- namely those whose polarizations are edge ideals of \say{whiskered graphs}~(\cite{V1990}) and more generally facet ideals of \say{grafted complexes}~(\cite{Fa2005}) -- are glicci~(see also~\cite{CFHNVT2025}).

 Our second contribution is to establish a framework for the investigation of polarization, a concept traditionally associated with monomial ideals, within the context of Gr\"obner bases of arbitrary homogeneous ideals in a polynomial ring. We do this with an eye towards extending \cref{cor:summaryMonomialCase} beyond the monomial setting. As Nagel and R\"omer \cite{NR08} described how to extract a basic double G-link from a vertex decomposition, Klein and Rajchgot \cite{KR21} described how to extract an elementary G-biliaison from what is called a geometric vertex decomposition (see \cref{def:gvd}), introduced by Knutson, Miller, and Yong \cite{KMY09}. A geometric vertex decomposition is a decomposition of an ideal that can be performed on a Gr\"obner basis with respect to a distinguished variable, but only if every element in the Gr\"obner basis is linear or constant as a polynomial in the distinguished variable. Having used polarization of monomial ideals to transfer information gleaned from Nagel and R\"omer's result to ideals that are not squarefree, we sought a comparable framework for transferring information gleaned from Klein and Rajchgot's result to ideals whose Gr\"obner bases are not linear in their respective distinguished variables.
 
 Towards this end, we introduce in \cref{sect:geoPol} geometric polarization, named in analogy to geometric vertex decomposition. Given a Gr\"obner basis $\G$ of an ideal in the polynomial ring, we define its geometric polarization with respect to a variable $y$, denoted $\P_y(\G)$, to be the set of polynomials obtained by replacing each power $y^n$ occurring in an element of $\G$ by $y(y')^{n-1}$, where $y'$ is a new variable (see \cref{def:geoPol}). Repeating this operation on the natural set of generators of a monomial ideal (which forms a Gr\"obner basis) recovers the usual notion of polarization. In the context of ideals that are not monomial ideals, further subtleties occur, such as the fact that $\P_y(\G)$ need not be a Gr\"obner basis of the ideal it generates. We characterize when geometric polarization produces a Gr\"obner basis as follows:
 \begin{introthm}[{\cref{thm:inducedGB}}]
      Fix a term order so that $y$ is lexicographically largest and a Gr\"obner basis $\G$ with respect to this order. The set $\P_y(\G)$ forms a Gr\"obner basis if and only if $y-y'$ is a nonzerodivisor modulo the ideal $(\P_y(\G))$.
 \end{introthm}
 
 Many fundamental features of the traditional notion of polarization extend to geometric polarization when $\P_y(\G)$ forms a Gr\"obner basis (see \cref{prop:stFacts}). Specifically, the heights of $(\G)$ and $(\P_y(\G))$ are equal, and the quotients they define are Cohen--Macaulay or not alike.  Moreover, if $(\G)$ is prime, then so too is $(\P_y(\G))$.
 
 After studying these properties, we give a statement relating elementary G-biliaisons of ideals to elementary G-biliaisons of their geometric polarizations (see in \cref{thm:polarLink}).  \cref{thm:polarLink} is the analogue of \cref{prop:liftPol}, described above and which concerns only monomial ideals, in the more general setting of homogeneous ideals in a polynomial ring.

\begin{introthm}[\cref{thm:polarLink}]
   If $\P_y(\G)$ is a Gr\"obner basis, then, under suitable hypotheses, the elementary G-biliaison induced from a geometric vertex decomposition of $(\P_y(\G))$ at $y$ descends to an elementary G-biliaison pertaining to $(\G)$.
\end{introthm}

Throughout this paper, we let $\kappa$ denote an arbitrary field.

\section{Monomial ideals and Gorenstein liaison: background}
\label{sec:background}

\subsection{Vertex decomposition and basic double G-links}\label{s:vertex-decomposition}

 Let $\Delta$ be a simplicial complex on vertex set $[n] = \{1,2,\dots, n\}$, i.e., a set of subsets of $[n]$ that is closed under taking subsets. An element $F \in \Delta$ is called a \textbf{face}.  The dimension of the face $F$ is $|F|-1$, and the dimension of $\Delta$ is the maximum dimension of any of its faces.  A face that is maximal under inclusion is called a \textbf{facet}, and a face of dimension $0$ is called a \textbf{vertex}.  We call $\Delta$ \textbf{pure} if all of its facets have the same dimension. We do not require $\{i\} \in \Delta$ for each $i \in [n]$.
 
There is a one-to-one correspondence between simplicial complexes on the vertex set $[n]$ and square-free monomial ideals in the polynomial ring $R= \kappa[x_1, \ldots, x_n]$, where  a simplicial complex $\Delta$
 corresponds to its {\bf Stanley--Reisner ideal} $$I_\Delta = \Big (\prod_{i \in U} x_i: U \subseteq [n], U \notin \Delta \Big ).$$  In this case $\Delta$ is called the {\bf Stanley--Reisner complex of} $I_\Delta$.  The minimal primes of $I_\Delta$ are generated by the variables corresponding to the complements of the facets of $\Delta$.  In particular, the Krull dimension of $R/I_{\Delta}$ is $\dim(\Delta)+1$, and $R/I_\Delta$ is equidimensional if and only if $\Delta$ is pure.  

For a simplicial complex $\Delta$ and vertex $i\in [n]$ not in  $\Delta$, we call $$\cone_\Delta(i)=\{F \cup A : F \in \Delta, \ A\subseteq \{i\}\}$$ 
the \textbf{cone over $\Delta$ with apex $i$.} Note that $$I_\Delta = I_{\cone_\Delta(i)}+(x_i).$$
 
Given a simplicial complex $\Delta$ and vertex $v\in \Delta$, define the following subcomplexes:
\begin{itemize}
\item the {\bf deletion} of $v$ is the set $\del_{\Delta}(v) = \{F\in \Delta: F\cap \{v\} = \emptyset\}$; 
\item the {\bf link} of $v$ is the set $\lk_{\Delta}(v) = \{F\in \del_{\Delta}(v): F\cup\{v\}\in \Delta\}$.
\end{itemize}

Note that $\lk_\Delta(v)$ and $\del_\Delta(v)$ are naturally complexes on $[n] \setminus \{v\}$.  When we write $I_{\lk_\Delta(v)}$ and $I_{\del_\Delta(v)}$, we will understand these to be the Stanley--Reisner ideals of $\lk_\Delta(v)$ and $\del_\Delta(v)$, respectively, computed in the ($n-1$)-dimensional polynomial ring $\kappa[x_1, \ldots, \widehat{x_v}, \dots, x_n]$.  Abusing notation slightly, we will typically write $\lk_\Delta(v)$ for $\lk_\Delta(v)R$ and $\del_\Delta(v)$ for $\del_\Delta(v)R$.

This gives the following equality: 
\begin{equation}\label{eq:linkdel}  I_{\Delta} = I_{\del_\Delta(v)}+x_vI_{\lk_\Delta(v)}.
\end{equation}

\begin{defn}
 \label{def:vertex decomposable}
A simplicial complex $\Delta$ is {\bf vertex decomposable} if $\Delta$ is pure and if either
\begin{enumerate}
    \item $\Delta = \{\emptyset\}$ or $\Delta$ is a simplex, or
    \item there is a vertex $v\in \Delta$ such that $\lk_\Delta(v)$ and $\del_{\Delta}(v)$ are vertex decomposable. 
\end{enumerate}
\end{defn}

If $\Delta$ is a cone with apex $v$, then $\lk_\Delta(v) = \del_\Delta(v)$. Otherwise, in situation $(2)$ of  \cref{def:vertex decomposable}, $\dim(\del_\Delta(v)) = \dim(\lk_\Delta(v))+1$, and we call $v$ a {\bf shedding vertex}.

\begin{example}\label{e:running} Let $\Delta$ be the simplicial complex with facets $\{1,4\}$, $\{2,3\}$, and $\{2,4\}$. Then 
$$I_\Delta = (x_1x_2,x_1x_3,x_3x_4)=(x_2,x_3)\cap (x_1,x_4) \cap (x_1,x_3).$$ 
Since all facets of $\Delta$ have dimension $1$, $\Delta$ is a pure simplicial complex. Then $\lk_\Delta(1)$ is the simplex with facet $\{4\}$, and $\Delta' =\del_\Delta(1)$ is the (pure) simplicial complex whose facets are $\{2,3\}$, $\{2,4\}$. Continuing, $\lk_{\Delta'}(4)$ is the simplex with facet $\{2\}$, and $\del_{\Delta'}(4)$ is the simplex with facet $\{2,3\}$.  Thus, $\Delta$ is vertex decomposable.

$$\begin{array}{cll}
&
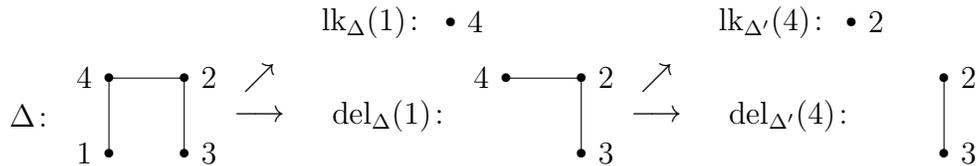
\begin{tikzpicture}
\tikzstyle{point}=[inner sep=0pt]
\node [point] at (-1,1) {$\lk_\Delta(1) \colon$};
\node (b)[point,label=right:$4$] at (0,1) {\tiny{$\bullet$}};
\end{tikzpicture}
&
\begin{tikzpicture}
\tikzstyle{point}=[inner sep=0pt]
\node [point] at (-1,1) {$\lk_{\Delta'}(4) \colon$};
\node (c)[point,label=right:$2$] at (0,1) {\tiny{$\bullet$}};
\end{tikzpicture} 
\\
\begin{tikzpicture}
\tikzstyle{point}=[inner sep=0pt]
\node at (-1,0.5) {$\Delta \colon$};
\node at (2,1) {$\nearrow$};
\node at (2,.5) {$\longrightarrow$};
\node (a)[point,label=left:$1$] at (0,0) {\tiny{$\bullet$}};
\node (b)[point,label=left:$4$] at (0,1) {\tiny{$\bullet$}};
\node (c)[point,label=right:$2$] at (1,1) {\tiny{$\bullet$}};
\node (d)[point,label=right:$3$] at (1,0) {\tiny{$\bullet$}};
\draw (a.center) -- (b.center);
\draw (b.center) -- (c.center);
\draw (c.center) -- (d.center);
\end{tikzpicture}
& 
\begin{tikzpicture}
\tikzstyle{point}=[inner sep=0pt]
\node at (-1.5,0) {$\del_\Delta(1) \colon$};
\node at (2,.5) {$\nearrow$};
\node at (2,0) {$\longrightarrow$};
\node (b)[point,label=left:$4$] at (0,.5) {\tiny{$\bullet$}};
\node (c)[point,label=right:$2$] at (1,.5) {\tiny{$\bullet$}};
\node (d)[point,label=right:$3$] at (1,-.5) {\tiny{$\bullet$}};
\draw (b.center) -- (c.center);
\draw (c.center) -- (d.center);
\end{tikzpicture}
&
\begin{tikzpicture}
\tikzstyle{point}=[inner sep=0pt]
\node at (-1,.5) {$\del_{\Delta'}(4) \colon$};
\node (c)[point,label=right:$2$] at (1,1) {\tiny{$\bullet$}};
\node (d)[point,label=right:$3$] at (1,0) {\tiny{$\bullet$}};
\draw (c.center) -- (d.center);
\end{tikzpicture} 
\end{array}
$$

By contrast, decomposing $\Delta$ at the vertex $4$ is not a valid step towards showing $\Delta$ to be vertex decomposable because $\del_\Delta(4)$ has facets $\{2,3\}$ and $\{1\}$ and is therefore not pure.
 
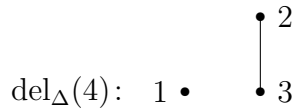
\[
 \begin{tikzpicture}
\tikzstyle{point}=[inner sep=0pt]
\node at (-1.5,0) {$\del_\Delta(4) \colon$};
\node (a)[point,label=left:$1$] at (0,0) {\tiny{$\bullet$}};
\node (c)[point,label=right:$2$] at (1,1) {\tiny{$\bullet$}};
\node (d)[point,label=right:$3$] at (1,0) {\tiny{$\bullet$}};
\draw (c.center) -- (d.center);
\end{tikzpicture} \qedhere
\]
\end{example}

\begin{defn} A pure $d$- dimensional simplicial complex $\Delta$ is {\bf shellable} if the  facets of $\Delta$ can be ordered as $F_1,\ldots, F_r$ such that for each $i\geq 1$, if $\Delta_i$ is the simplicial complex with facets $F_1, \ldots, F_i$, then $\Delta_i \cap F_{i+1}$  is pure of dimension $d-1$.
\end{defn}

If $\Delta$ is vertex decomposable, then $\Delta$ is Cohen--Macaulay (i.e., the quotient by its Stanley--Reisner ideal is Cohen--Macaulay).  Specifically, every vertex decomposable simplicial complex is shellable (see\cite{PB80}, where vertex decomposability was introduced), and every shellable simplicial complex is Cohen--Macaulay, which can be deduced from Reisner's criterion \cite{Rei76}, a topological criterion on $\Delta$ characterizing when $\Delta$ is Cohen--Macaulay.

\begin{defn}\label{d:wvd}
A simplicial complex $\Delta$ is {\bf weakly vertex decomposable} if $\Delta$ is pure and if either
\begin{enumerate}
    \item $\Delta = \{\emptyset\}$ or $\Delta$ is a simplex, or
    \item  there is a vertex $v\in \Delta$ such that $\lk_\Delta(v)$ is vertex decomposable and $\del_{\Delta}(v)$ is Cohen--Macaulay. 
\end{enumerate}
\end{defn}

In situation $(2)$ of \cref{d:wvd}, if $\Delta$ is not a cone with apex $v$, then $\dim(\del_\Delta(v)) = \dim(\lk_\Delta(v))+1$, and we call $v$ a {\bf weak shedding vertex}.

From the fact that a vertex decomposable simplicial complex is Cohen--Macaulay, it follows that every vertex decomposable simplicial complex is weakly vertex decomposable.

In order to describe the relationship between vertex decomposition and Gorenstein liaison, we now require the notion of a basic double G-link, for which we first recall some standard ring theoretic definitions.

Let $R$ be a ring and $I$ an ideal of $R$. We say that $R/I$ is \textbf{generically Gorenstein}, or $\mathbf{G_0}$, if $(R/I)_P$ is Gorenstein for every minimal prime $P$ of $I$. We say that $I$ is \textbf{unmixed} if $\hgt(I) = \hgt(P)$ for every associated prime $P$ of $I$.  If $R$ is the homogeneous coordinate ring of projective space and $I$ is the saturated defining ideal of the scheme $X$, then $I$ is unmixed if and only if $X$ is equidimensional and has no embedded components.  Recall that a Cohen--Macaulay ideal is necessarily unmixed.

 We recall the definition of Gorenstein linkage of homogeneous ideals of a polynomial ring.

\begin{defn}\label{def:GorensteinLikage}
    Let $R$ be a polynomial ring over a field, and let $I_{C_1}$ and $I_{C_2}$ be homogeneous, saturated ideals of $R$. We say that $I_{C_1}$ and $I_{C_2}$ are \textbf{directly Gorenstein linked (or G-linked)} by a homogeneous, saturated ideal $I_X$ if $I_X \subseteq I_{C_1} \cap I_{C_2}$, $R/I_X$ is Gorenstein, $I_X:I_{C_1} = I_{C_2}$, and $I_X:I_{C_2} = I_{C_1}$.  
\end{defn}

We will be interested in applying \cref{def:GorensteinLikage} to unmixed ideals.  If $I_{C_1}$ is unmixed, then $I_X:I_{C_1} = I_{C_2}$ implies $I_X:I_{C_2} = I_{C_1}$ and $I_{C_2}$ is also unmixed.  Let $C_1$, $C_2$, and $X$ denote the subschemes of projective space defined by $I_{C_1}$, $I_{C_2}$, and $I_X$, respectively.  Then it is the same to say that $I_{C_1}$ and $I_{C_2}$ are G-linked by $I_X$ and that $C_1$ and $C_2$ are G-linked by $X$.

\begin{defn}\label{def:basic double G-link} 
Let $R$ be a polynomial ring over a field, and let $f$ be a homogeneous degree $d$ element of $R$.  Let $A\subset B$ be homogeneous, unmixed proper ideals. Then the ideal $fB+A$ is called a \textbf{basic double G-link of $B$ on $A$ of shift $d$} if  
\begin{itemize}
    \item $R/A$ is Cohen--Macaulay and $G_0$;
    \item $\hgt(A) = \hgt(B)-1$; and
    \item $A:f=A$.
\end{itemize}
   If $A$ and $B$ are monomial ideals and $f$ is a monomial, we call $f B+A$ a \textbf{monomial basic double G-link of $B$ on $A$}. 

   As the name suggests, if $C$ is a basic double G-link of $B$, then $B$ is G-linked to $C$ in two steps \cite[Proposition 5.10]{KMM+01}, i.e., there is an ideal $D$ so that $B$ is directly G-linked to $D$ and $D$ is directly G-linked to $C$. 

Abusing notation, we will sometimes refer to an equality $C = fB+A$ as a basic double G-link to mean that $C$ is a basic double G-link of $B$ on $A$.
\end{defn}

For a discussion of basic double G-linkage in terms of linear equivalence of divisors, see \cite[Section 4]{KMM+01}.

\begin{example}\label{e:running-2}
    Continuing with $I_\Delta =(x_1x_2,x_1x_3,x_3x_4)$ 
    as in \cref{e:running}, we set $A=(x_3x_4)$ and $B=(x_2,x_3)$. Then $$I_\Delta = x_1(x_2,x_3)+(x_3x_4)$$
     is a monomial basic double G-link of $(x_2,x_3)$ on $(x_3x_4)$ of shift $1 = \deg(x_1)$. Observe that $(x_2,x_3)$ is the Stanley--Reisner ideal of $\lk_\Delta(1)$ and that $(x_3x_4)$ is the Stanley--Reisner ideal of 
     $\del_\Delta(1)$.
        We will return to $I_\Delta$ and its decomposition at $x_1$ (with slightly different notation) again in \cref{ex:monomialPol}.
\end{example}

In \cite{NR08}, Nagel and R\"omer  completely characterized the monomial basic double G-links of shift $1$ of Stanley--Reisner ideals.  Specifically, they showed that, if $\Delta$ is pure and if $\del_{\Delta}(v)$ is Cohen--Macaulay and has the same dimension as $\Delta$, then $I_\Delta$ is a basic double $G$-link of the cone over the Stanley--Reisner ideal of $\lk_{\Delta}(v)$.  Conversely, if there exists a vertex $v$ of $\Delta$ so that $I_\Delta = x_v I_{\lk_{\Delta}(v)}+I_{\del_{\Delta}(v)}$ is a basic double $G$-link, then $v$ is a weak shedding vertex of $\Delta$. It follows (see \cite[Theorem 3.3]{NR08}) that the Stanley--Reisner ideal of a weakly vertex decomposable simplicial complex $\Delta$ is G-linked to an ideal generated by indeterminates via a sequence of monomial basic double G-links of shift $1$, and, in particular, that the quotient by that ideal is Cohen--Macaulay.

\subsection{Polarization}

Polarization is a method for transforming a monomial ideal in a polynomial ring into squarefree monomial ideal in an enlarged polynomial ring, while preserving many of the algebraic properties of the original ideal.

Hartshorne~\cite{Har66} introduced the technique of \emph{distraction} and a related auxiliary construction, now known as polarization, in the course of proving connectedness of the Hilbert scheme. Polarization was later rediscovered independently by Fr\"oberg~\cite{Fr} and Weyman~\cite{W}, possibly among others.

To motivate the definition, we give an example of an ideal $I$ and its polarization, $\P(I)$: 
\begin{equation}\label{e:simple}
I=(x^2y^3,y^2z,xz) \quad \P(I)=(x_1x_2y_1y_2y_3,y_1y_2z_1,x_1z_1).
\end{equation}

The formal definition of polarization is below.

 \begin{defn}
 \label{def:polarization}
 Given a monomial ideal $I$ of a polynomial ring $R=\kappa[x_1,\ldots, x_n]$ with minimal monomial generating set $G(I)$,
 the {\bf polarization} of $I$ is the squarefree monomial ideal  
 \[
 \P(I)=\left (\prod_{i=1}^n\prod_{j=1}^{a_i}x_{i,j} \colon \prod_{i=1}^n x_i^{a_i}\in G(I)\right)
 \]
 in the ring $S=\kappa[x_{i,j}\mid 1\leq i\leq n, 1\leq j\leq e_i]$, where $e_i$ is the maximum exponent of $x_i$ that appears in an element of $G(I)$.  
 
 A {\bf partial polarization} of $I$ with respect to the vector $\bb=(b_1,\ldots,b_n) \in \ZZ_+^n$ is the ideal  
 \[
 \P_{\bb}(I)=\left ( \prod_{i=1}^n\left(x_{i,{b_i}}^{\max\{0,a_i-b_i\}}\prod_{j=1}^{\min\{a_i,b_i\}}x_{i,j} \right) : \prod_{i=1}^n x_i^{a_i}\in G(I)\right)
 \]
 in the ring $S=\kappa[x_{i,j}: 1\leq i\leq n, 1\leq j\leq e_i]$, where $e_i$ is either the maximum exponent of $x_i$ that appears in any element of $G(I)$ or $b_i$, whichever is less.  
\end{defn}

For example, the partial polarization of the ideal $I$ in \cref{e:simple} with respect to the vector $b=(2,2,3)$ is 
$$\P_{\bb}(I)=(x_1x_2y_1y_2^2,y_1y_2z_1,x_1z_1).$$

\noindent In general, $$I=\P_{(1,\ldots, 1)}(I),$$ up to a relabeling of the variables,  and 
$$\P(I)=\P_{\bb}(I) \qwhere b_i = \max \{a \colon x_i^a \mid \mu \mbox{ for some } \mu \in G(I)\} \qforall i \in [n].$$

It is worth noting that while the polarization of a monomial ideal is unique (up to a relabeling of variables), many non-isomorphic monomial ideals may have the same polarization. For example, the ideal $\P(I)$ appearing in \cref{e:simple} is also, up to relabeling,  the polarization of $(a^2b^2c, b^2d, ad)$ and $(uvwt^2,wtz,uz)$, among others. Specifically, the reverse operation \say{depolarization} is not well-defined except when one retains information about $I$ (as in  \cref{not:depolarization}).

Because there is a regular sequence $f_1, \ldots, f_k$ of elements of the form $x_i-x_{i,j}$ so that $$R/I \cong S/(\P(I)+(f_1, \ldots, f_k)),$$ $R/I$ is a complete intersection, Gorenstein or Cohen--Macaulay if and only if $S/\P(I)$ is \cite{Fr} (hence, if and only if any $S/\P_{\bb}(I)$ is).  Moreover, $\hgt(I) = \hgt(\P(I))$, and $(x_{i_1}, \ldots, x_{i_r})$ is an associated prime of $I$ if and only if some $(x_{i_1,j_1}, \ldots, x_{i_r,j_r})$ is an associated prime of $\P(I)$ by \cite[Proposition 4.4]{Har66} or \cite[Proposition 2.3]{Far06}.

\section{Polarization and basic double G-links}\label{s: pol and liaison}
 We begin this section by expressing monomial ideals in a form appropriate for study via basic double G-links. 
 
 \begin{notation}\label{not:minimal-monomial-generators}
     If $I$ is a monomial ideal of a polynomial ring, let $G(I)$ denote the set of minimal monomial generators of $I$.
 \end{notation}
 
\begin{lemma}\label{lem:formOfBDL}
Suppose that $A$, $B$, and $C$ are monomial ideals of the polynomial ring $R = \kappa[x_1, \ldots, x_n]$ and that $z = x_i$ for some $i \in [n]$.  The following two sets of conditions are equivalent:
\begin{enumerate}
    \item $C = zB+A$, $A:z =A$, and $A\subseteq B$;
    \item  $B = C:z$ and $A = (\sigma \in G(C) : z \nmid \sigma)$.
\end{enumerate}
In particular, the ideals $A$ and $B$ satisfying the conditions in (1) are uniquely determined by $C$. Moreover, when these conditions are satisfied, $A=B$ if and only if $z$ does not divide any element of $G(C)$.  
 \end{lemma}
\begin{proof}
Assume first that the three conditions in (1) are satisfied.  Because $A$, $B$, and $C$ are all monomial ideals, their sets of monomial generators are related by 
\[
G(C) = G(A) \cup \{z\mu : z\mu \notin A, \mu \in G(B)\}.
\]

Since  $A:z =A$, the elements of $G(A)$ are not divisible by $z$. Thus, $A = (\sigma \in G(C) : z \nmid \sigma)$.  Clearly, $B \subseteq C:z$.  Fix $\gamma \in G(C:z)$.   Because $\gamma$ is a minimal generator of $C:z$, we know if $\gamma \in C$, then $z \nmid \gamma$,  and so $\gamma \in A$. On the other hand if $\gamma \notin C$, then $\gamma \notin A \subset C$ and so $z\gamma \notin A$ because $A:z = A$.  With $z\gamma \in G(C) \setminus G(A)$, we must have $z\gamma \in \{z\mu : z\mu \notin A, \mu \in G(B)\}$, and so $\gamma \in G(B)$.  Thus, $C:z \subseteq B$, as well. This establishes (2).

For the converse implication, $(2)\Rightarrow (1)$, it is clear that $C = z(C:z)+(\sigma \in G(C) : z \nmid \sigma)$, i.e., $C = zB+A$.
The identity $A:z=A$ follows since the minimal generators of $A$ are not divisible by $z$ and the containment $A\subseteq B$ follows since $A\subseteq C\subseteq B$.

The final sentences are immediate from the description of $A$ and $B$ in (2).
\end{proof}

The next lemma considers how a polarization $\P_{\bb}(I)$ of a monomial ideal $I$ might occur as a monomial basic double G-link of shift $1$. We will show that any such basic double G-link may be assumed to be taken at $x_{1,1}$. Viewed through the lens of \cite{NR08}, \cref{lem:noBadPols} says that any weak vertex decomposition of the Stanley--Reisner complex of the full polarization $\P(I)$ may be assumed to be taken at the vertex corresponding to $x_{1,1}$.

\begin{lemma}\label{lem:noBadPols}
Let $I$ be a monomial ideal of the polynomial ring $R = \kappa[x_1, \ldots, x_n]$, and fix $i \in [n]$. Set 
\[
k =  \max \{a: x_i^a \mid \mu \qforsome  \mu \in G(I)\}, 
\] and fix some $1<a \in [k]$.  Suppose that there exists $\nu \in G(I)$ so that $$x_i \mid \nu \qand x_i^a \nmid \nu.$$  
Let $\P_{\bb}(I)$ denote a partial polarization of $I$ with respect to the vector $\bb \in \ZZ_+^n$ satisfying $a \leq b_i \leq k.$  Set $$A = (\sigma \in G( \P_{\bb}(I)) : x_{i,a} \nmid \sigma) \qand B = \P_{\bb}(I):x_{i,a}.$$  Then $\P_{\bb}(I) = x_{i,a}B+A$ is not a basic double G-link.  
 \end{lemma}
 
 \begin{proof} Suppose, for contradiction, that $\P_{\bb}(I) = x_{i,a}B+A$ is a basic double G-link.  By assumption, there is some generator $\nu$ of $G(I)$ so that $x_i \mid \nu$ but $x_i^a \nmid \nu$.  This generator $\nu$ of $I$ gives rise to a generator $\mu$ of $A$ so that $x_{i,1} \mid \mu$.  Because the quotient by $A$ is Cohen--Macaulay by the definition of basic double G-link, $A$ must have some minimal prime $P$ so that $x_{i,1} \in P$ and $\hgt(P) = \hgt(A)$.  

Fix $\lambda \in G(B)$, and note that $x_{i,a} \nmid \lambda$.  Thus, if $\lambda \in \P_{\bb}(I)$, then $\lambda \in A \subseteq P$.  If $\lambda \notin \P_{\bb}(I)$, then $x_{i,a}\lambda$ is a minimal generator of $\P_{\bb}(I)$, and so $x_{i,1} \mid \lambda$.  Hence $\lambda \in P$.  Because $\lambda$ was an arbitrary generator of $B$, $B \subseteq P$, and so $\hgt(B) \leq \hgt(P) = \hgt(A)$.  But this contradicts the definition of basic double G-link, which would require $\hgt(B) = \hgt(A)+1$.
 \end{proof}

Note that the condition $x_i \mid \nu$ but $x_i^a \nmid \nu$ in \cref{lem:noBadPols} exists merely to rule out trivialities of relabeling.  For example, if $I = (x_1^2, x_2)$, then $\P(I) = x_{1,2}(x_{1,1}, x_{2})+(x_2)$ is a basic double G-link even though $2>1$.  When $x_i^a$ divides every element of $G(I)$ divisible by $x_i$, then $x_{i,1}$ and $x_{i,a}$ are not meaningfully distinct. 

 Though we do not study this notion in our paper, there are cohomological properties that can be deduced from an equation of the form $C = xB+A$ where the quotient by $A$ is Cohen--Macaulay but not necessarily $G_0$ and $B$ is not necessarily unmixed.  Under these conditions, one says that $C$ is a basic double link (rather than basic double G-link) of $B$ on $A$.  \cref{lem:noBadPols} would still hold (by the proof given) if we replaced \say{is not a basic double G-link} with \say{is not a basic double link} in the conclusion.

We introduce a criterion which allows us to lift a monomial basic double G-link from a polarization $\P_{\bb}(I)$ of a monomial ideal $I$ to $I$ itself.

In order to do so, we first give a small lemma.

\begin{lemma}\label{lem:polarization-in-disguise}
Suppose that $B$ is a monomial ideal of the polynomial ring $R = \kappa[x_1, \ldots, x_n]$, and fix $\bb \in \ZZ_+^n$. Set $\B = \P_{\bb}(x_iB):x_{i,1}$.  Then $\B$ is obtained from $B$ by partial polarization and, possibly, relabeling of the variables.
\end{lemma}
\begin{proof}
If $b_i = 1$, then $\P_{\bb}(x_iB) = x_{i,1}\P_{\bb}(B)$, and so $\P_{\bb}(B) = \B$.  If $b_i>1$, set $\tilde{\bb} = \bb-\be_i \in \ZZ_+^n$ (where $\be_i$ denotes the $i^{th}$ standard basis vector).  Let $\tilde{B}$ be the ideal obtained from $\P_{\tilde{\bb}}(B)$ by the substitutions $x_{i,j} \mapsto x_{i,j+1}$ for all $j$.  Then $\P_{\bb}(x_iB) = x_{i,1}\tilde{B}$, and so $\B = \tilde{B}$.
\end{proof}

 \begin{proposition}\label{prop:liftPol}
 Let $I$ be a monomial ideal of the polynomial ring $R = \kappa[x_1, \ldots, x_n]$, and let $\P_{\bb}(I)$ be its polarization with respect to the vector $\bb \in \ZZ_+^n$.  
 If for some $r,s \in [n]$,  $x_{r,s}$ determines a monomial basic double G-link 
 \[
 \P_{\bb}(I) = x_{r,s}\B + \A,
 \]
 then $x_r$ determines a decomposition of $I$ of the form
 $$ I = x_rB+A$$
 in which $\A=\P_{\bb}(A)$ and $\B = \P_{\bb}(x_rB):x_{r,s}$. If, moreover, the quotient by $A$ is $G_0$, then $I$ is a monomial basic double G-link of $B$ on $A$.
\end{proposition}
 
 \begin{proof}
 We assume, without loss of generality, that $r=1$. By \cref{lem:noBadPols}, we may also assume $s=1$.
 
 Set $$A = I:x_1 \qand B=(\sigma \in G(I) : x_1 \nmid \sigma).$$ 
 By \cref{lem:formOfBDL}, $$I = x_{1}B+A, \quad A\subset B, \qand A = A:x_1.$$
Now polarization gives 
\[ \P_{\bb}(I)= \P_{\bb}(x_1 B)+\P_{\bb}(A)= x_{1,1}\P_{\bb}(B)'+\P_{\bb}(A),\]
where $\P_{\bb}(B)' = \P_{\bb}(x_1B):x_{1,1}$. Since $A\subset B$, we have that $\P_{\bb}(A)\subset \P_{\bb}(B)$, and, since the generators of $\P_{\bb}(A)$ do not involve the variables $x_{1,i}$, it is also the case that $\P_{\bb}(A)\subset \P_{\bb}(B)'$.  By \cref{lem:formOfBDL} $\B=\P_{\bb}(B)'$ and $\A=\P_{\bb}(A)$.  
 
By \cref{lem:polarization-in-disguise}, $\B =\P_{\bb}(x_1B):x_{1,1}$ is, up to possibly relabeling the variables, a polarization of $B$. Since polarization preserves height, unmixedness, and the Cohen--Macaulay property, we deduce that $\hgt(A)=\hgt(\A)$ and $\hgt(B)=\hgt(\B)$.  The hypothesis grants that the quotient by $\A$ is Cohen--Macaulay and that $\hgt(\B)=\hgt(\A)+1$; therefore, the quotient by $A$ is Cohen--Macaulay and $\hgt(B)=\hgt(A)+1$. Hence, $I$ is a basic double G-link of $B$ on $A$ whenever the quotient by $A$ is $G_0$.
 \end{proof}

 \begin{example}\label{ex:monomialPol}
Consider the ideals $I = (x,y)^2 = (x^2,xy,y^2)$ and \[
J = (x,y,z)^2 = (x^2, xy, xz, y^2, yz, z^2)\subset R=\kappa[x,y,z]
\] and their polarizations $\P(I) = (x_1x_2,x_1y_1,y_1y_2)$ and $\P(J) = (x_1x_2, x_1y_1, x_1z_1, y_1y_2, yz, z_1z_2)$ in the polynomial ring $S=\kappa[x_1,x_2,y_1,y_2,z_1,z_2]$. 
In $S$, consider the basic double G-link
\[
\P(I) = x_1(x_2,y_1)+(y_1y_2).
\]
Then $\P((y^2))=(y_1y_2)$ and $\P((x,y))=(x_1,y_1)$, from which $(x_2,y_1)$ is obtained by the substitution $x_1 \mapsto x_2$.
In $R$, because the quotient by $(y_1^2)$ is $G_0$, one has the corresponding basic double G-link
\[
I = x(x,y)+(y^2).
\]
By contrast, the basic double G-link 
\[
\P(J) = x_1(x_2,y_1, z_1)+(y_1y_2, y_1z_1, z_1z_2) 
\] gives rise to the equation \[
J = x(x,y,z)+(y^2,yz,z^2),
\] but this equation does not constitute a basic double G-link because $(y^2,yz,z^2)$ does not define a $G_0$ quotient.

In this case, the failure of $(y^2,yz,z^2)$ to be $G_0$ can be addressed by Migliore and Nagel's lifting construction \cite{MN00}, which shows that $J$ and $(x,y,z)$ are connected by basic double $G$-link. The purpose of this example is to show that a basic double G-link involving a polarization of an ideal may not directly induce a basic double G-link involving the original ideal.
\end{example}

We now work to give a converse to \cref{prop:liftPol}.  We begin by considering the $G_0$ condition.

\begin{lemma}\label{lem:G0pol}
    Suppose that $A$ is a monomial ideal of $R$ and that $R/A$ is equidimensional and $G_0$ and that $\P_{\bb}(A)$ is the partial polarization of $A$ with respect to the vector $\bb \in \ZZ_+^n$.  Let $S$ denote the ambient polynomial ring of $\P_{\bb}(A)$ as described in \cref{def:polarization}.  Then $S/\P_{\bb}(A)$ is also $G_0$.
\end{lemma}
\begin{proof}   
Fix a minimal prime $P = (x_{1,j_1}, \ldots, x_{h,j_h})$ of $\P_{\bb}(A)$.  Note that $A$ and $\P_{\bb}(A)$ share a (full) polarization, which we will call $J$.  Using \cite[Proposition 4.4]{Har66} or \cite[Corollary 2.6]{Far06}, $Q = (x_1, \ldots, x_h)$ is an associated prime of $A$.  Because $A$ and $\P_{\bb}(A)$ are of the same height by \cite[Proposition 4.4]{Har66} or \cite[Proposition 2.3]{Far06} and $R/A$ is equidimensional, $Q$ is a minimal prime of $A$.  

Hence, by the assumption that $R/A$ is $G_0$, $(R/A)_Q$ is Gorenstein.  That is, $AR_Q$ is a Gorenstein artinian monomial ideal.  But then $A$ must be a complete intersection monomial ideal \cite[Proposition on Page 2]{Bei93}, which is to say that the generators of $AR_Q$ have pairwise disjoint support.  So too, then, do the generators of $\P_{\bb}(A)S_P$; if any two of the minimal generators of $\P_{\bb}(A)$ were divisible by some $x_{i,j_i}$, then the corresponding generators of $AR_Q$ would both be divisible by $x_i$.  Thus, $\P_{\bb}(A)S_P$ is a complete intersection monomial ideal.  In particular, $(S/\P_{\bb}(A))_P$ is Gorenstein, as desired.
\end{proof}

\begin{proposition}\label{prop:BDLgivesBDLofPolarization}
Let $I$ be a monomial ideal of the polynomial ring $R = \kappa[x_1, \ldots, x_n]$.  Suppose that $I = x_iB+A$ is a basic double G-link of $B$ on $A$, and fix $\bb \in \ZZ_+^n$.  Set $\B = \P_{\bb}(x_iB):x_{i,1}$.  Then $\P_{\bb}(I) = x_{i,1}\B+\P_{\bb}(A)$ is a basic double G-link of $\B$ on $\P_{\bb}(A)$.
\end{proposition}

\begin{proof}
From $I = x_iB+A$, we have the equation $\P_{\bb}(I) = \P_{\bb}(x_iB)+\P_{\bb}(A) = x_{i,1}\B+\P_{\bb}(A)$.

All unmixedness, Cohen--Macaulayness, and height properties are preserved under polarization.  By \cref{lem:polarization-in-disguise}, $\B$ is, up to relabeling of the variables, a polarization of $B$, and $\P_{\bb}(A)$ is by definition a polarization of $A$. 
 By \cref{lem:G0pol}, the quotient by $\P_{\bb}(A)$ is $G_0$.

 It remains only to show that $\P_{\bb}(A) \subset \B$. 
 Because no generator of $A$ is divisible by $x_i$ and $A \subset B$, each generator of $A$ is divisible by some element of $B$ that is not divisible by $x_i$.  Thus, every element of $\P_{\bb}(A)$ is divisible by some element of $\P_{\bb}(B)$ that is not divisible by $x_{i,j}$ for any $j$ and, hence, is also an element of $\B$.  That is, $\P_{\bb}(A) \subset \B$.
\end{proof}

\begin{notation}\label{not:depolarization} Given a monomial ideal $I$ of the polynomial ring $R = \kappa[x_1, \ldots, x_n]$, its partial polarization $\P_{\bb}(I)$ with respect to the polarization vector $\bb \in \ZZ_+^n$, and an ideal $A \subseteq \P_{\bb}(I)$, we will call the ideal \[
(A+(x_i - x_{i,j} \colon i \in [n], j \in [b_i])) \cap R
\]the \textbf{depolarization} of $A$.  Note  that the depolarization of $\P_{\bb}(I)$ is $I$ and that the depolarization of $A$ is contained in $I$.
\end{notation}

Given monomial ideals $A$, $B$, and $C$ and a monomial basic double G-link $C = xB+A$ of shift $1$, call $A$ the \textbf{ideal of the deletion}.  The terminology is justified by \Cref{cor:summaryMonomialCase} and the reference \cite[Remark 2.4]{NR08} cited in its proof.  

\begin{theorem}\label{cor:summaryMonomialCase}
Let $I$ be a monomial ideal in a polynomial ring, $\P_{\bb}(I)$ its polarization with respect to the vector $\bb$, and $\P(I)$ its full polarization. The following are equivalent:
\begin{enumerate}
\item $I$ is glicci via a sequence monomial basic double G-links of shift $1$;
\item $\P_{\bb}(I)$ is glicci via a sequence of monomial basic double G-links of shift $1$, and the depolarizations of the ideals of the deletions define $G_0$ quotients;
\item $\P(I)$ is the Stanley-Reisner ideal of a weakly vertex decomposable simplicial complex, and the depolarizations of the Stanley--Reisner ideals of the deletions appearing in each weak vertex decomposition define $G_0$ quotients.
\end{enumerate}
\end{theorem}
\begin{proof}
The equivalence of (2) and (3) follows from \cite[Remark 2.4]{NR08}. The equivalence of (1) and (2) follows from \cref{prop:liftPol,prop:BDLgivesBDLofPolarization}.
\end{proof}

We now apply the results from this section to monomial ideals that have appeared as the initial ideals of ideals of independent interest.  

\begin{example}\label{ex:manySquares}
Let $I = (x^2-yz,wz^2-y^2x,wxz-y^3)$. One can check that $I$ is a radical ideal that defines a Cohen--Macaulay quotient ring and yet has no squarefree initial ideal. (It is also not weakly geometrically vertex decomposable in the sense of \cite[Definition 4.6]{KR21}.)  With respect to the lexicographic term order on $w>x>y>z$, the given generators form a Gr\"obner basis.  The corresponding initial ideal is $J = (x^2,wz^2,wxz)$, whose polarization $\P(J)=(xx',wzz',wxz)$ is the Stanley--Reisner ideal of a vertex decomposable simplicial complex.  If one first decomposes at $x$ and then at $z$, for example, then the depolarizations of the ideals of the respective deletions define $G_0$ quotients.  Via \cref{cor:summaryMonomialCase}, the vertex decomposition of the Stanley--Reisner complex of $\P(J)$ gives rise to a sequence of basic double G-links connecting $J$ to a complete intersection.  

Moreover, the first step of the geometric vertex decomposition of $J$ may be taken at whichever of $x,z,$ or $w$ one chooses, though at neither of $x'$ nor $z'$, as predicted by \cref{lem:noBadPols}.

In order to study $I$ itself, we will require a more general notion of polarization, which we develop in \cref{sect:geoPol}.  Having developed the appropriate machinery, we will show that $I$ is glicci in \cref{ex:mainTheroemApp1}.
\end{example}

\begin{example}
Set $R = \QQ[a,\ldots,i]$, and consider the toric ideal 
\begin{eqnarray*}
I =&(fg-eh, ah^2-dgi, agh-cfi, afh-dei, ag^2-cei, aeg-bch, af^2-bdh, aef-bdg, \\ 
& ae^2-bcf, 
 bh^2-efi, bgh-e^2i, dg^2-cfh, cf^2-deg, a^2eh-bcdi)
\end{eqnarray*}
of \cite[Example 13.17]{Stu96}, which defines a projectively normal toric variety yet has no squarefree initial ideals.  Sturmfels initially conjectured that none of the initial ideals of $I$ defined a Cohen--Macaulay quotient though later found that some did \cite[Section 4]{Stu00}.  A modern Macaulay2 \cite{M2} computation now shows that, among the $1879$ initial ideals of $I$, $28$ define Cohen--Macaulay quotients, all of them moreover having polarization that are the Stanley--Reisner ideals of vertex decomposable simplicial complexes. 

One such initial ideal is \[
J =\left(dg^{2},\,bh^{2},\,eh,\,deg,\,bdh,\,bdg,\,ae^{2},\,dgi,\,bch,\,e^{2}i,\,dei,\,cei,\,bcdi,\,cfi\right)
,\] whose polarization is \[
\P(J)=\left(dg_1g_2,\,bh_1h_2,\,eh,\,deg,\,bdh,\,bdg,\,ae_1e_2,\,dgi,\,bch,\,e_1e_2i,\,dei,\,cei,\,bcdi,\,cfi\right).
\] 

\noindent There is a decomposition $\P(J) = eB+A$, with $A=\left(dgi,\,cfi,\,bh_1h_2,\,bdh,\,bch,\,dg_1g_2,\,bdg,\,bcdi\right)$ and $B= A+(h,\,e_2i,\,di,\,ci, \,dg, \,ae_2)$. By \cref{prop:liftPol} this gives rise to a decomposition of $J$ as
\[
J=eB'+A' \qwith A'=\left(dgi,\,cfi,\,bh^2,\,bdh,\,bch,\,dg^2,\,bdg,\,bcdi\right), B'=(h,\,ei,\,di,\,ci, \,dg, \,ae).
\] 
Since the quotient by $A'$ is $G_0$, the above decomposition constitutes a monomial basic double G-link, so $J$ is G-linked to $B'$.
Continuing, $B'=\P(B')$ is G-linked to the complete intersection ideal $\left(i,\,h,\,g,\,ae\right)$ in two steps using the decomposition
\[
B'=d \left(i,\,h,\,g,\,ae\right) + \left( h, \, ei, \, ci,\, ae \right).
\]

By contrast, consider a different initial ideal \[
K =\left(eh,\,e^2i,\,dei,\,bh^2,\,ah^2,\,agh,\,bdh,\,dg^2,\,ag^2,\,deg,\,aeg,\,bdg,\,ae^2,\,bcdi\right)
\]
 of $I$. Its polarization \[\P(K) =\left(eh,\,e_1e_2i,\,dei,\,bh_1h_2,\,ah_1h_2,\,agh,\,bdh,\,dg_1g_2,\,ag_1g_2,\,deg,\,aeg,\,bdg,\,ae_1e_2,\,bcdi\right)\]
is associated to the vertex decomposable simplicial complex $\Delta_{\P(K)}$.  In contrast to the example involving $J$, the ideal of the deletion at any valid geometric vertex decomposition of $\Delta_{\P(K)}$ does not define a $G_0$ quotient, and so these vertex decompositions do not induce basic double G-links involving $K$.
\end{example}

In \cref{cor:summaryMonomialCase} we have found a close relationship between a strategy for establishing that a monomial ideal is glicci and a well-studied strategy for establishing that its polarization, a squarefree monomial ideal, is glicci, namely sequences of basic double G-links of shift $1$. We note below that the licci property can also, under appropriate hypotheses, be transferred between a monomial ideal and its polarization. Recall that an ideal is called \textbf{licci} if it is in the CI-liaison class of a complete intersection.  Context for this study was described in \cref{sec:intro}.

\begin{remark}\label{rem: licci}
    Let $I$ be a monomial ideal in a polynomial ring $R$ over an infinite field, and let $\bb$ be a vector of non-negative integers. If $I$ is equigenerated or if non-homogeneous links are permitted, then $I$ is licci if and only if its partial polarization $\P_{\bb}(I)$ is licci.  This equivalence follows from the combination of \cite[Proposition 3]{HU07} and \cite[Theorem 2.12]{HU2}.
\end{remark}

\subsection{Vertex decomposability of polarized  ideals}

In this section we give an application of the interplay between G-liaison and polarization 
developed in the previous section. 
We begin by recalling the notion of strongly stable monomial ideals.  These are some of the most important classes of monomial ideals in computational algebra since, e.g., in characteristic zero generic initial initials are strongly stable.  Indeed, every strongly stable monomial ideal is Borel fixed, and, in characteristic $0$, every Borel-fixed monomial ideal is strongly stable.

    A monomial ideal $I \subset S$ is called {\bf strongly stable} if one has $x_iu/x_j \in I$ for all monomials $u \in I$ and all $i < j$ such that $x_j$ divides $u$.

Stable ideals are a larger class of monomial ideals endowed with similar properties as strongly stable ideals (see \cite{EK}).  For a monomial $u$ we denote by $m(u)$ the largest index $j$ such that $x_j\mid u$. A monomial ideal $I$ is called a {\bf stable ideal} if for all monomials $u \in I$ and all $i < m(u)$ one has $x_iu/x_{m(u)} \in I$.

Migliore and Nagel \cite[Theorem 3.5]{MN02} showed that strongly stable ideals are glicci using basic double G-linkage, but not basic double G-linkage where all ideals involved are monomial. Therefore our \cref{cor:summaryMonomialCase} does not apply directly to show that polarizations of such ideals are glicci. However, we will prove that polarizations of stable ideals are glicci and that their simplicial complexes are vertex decomposable in \cref{thm:StableVertexDecomp}. 
 
A simplicial complex $\Delta$ is called {\bf shifted} if for all $F \in \Delta$, $j \in F$ and $j < i$ such that $i \not\in F$ we have $(F \setminus \{j\}) \cup \{i\} \in \Delta$. The study of shifted simplicial complexes is a combinatorial counterpart to the study of generic initial ideals (see \cite{AHH98}). Nagel and R\"omer \cite{NR08} showed that shifted simplicial complexes are glicci via a sequence of monomial basic double G-links of shift $1$. Although shifted simplicial complexes are analogous to generic initial ideals, polarizations of strongly stable ideals need not give rise to shifted complexes, as shown by the following example.

\begin{example}
Consider the strongly stable ideal $I=(x^2,xy, y^2)$ with polarization $\P(I)=(x_1x_2,x_1y_1,y_1y_2)$. The simplicial complex $\Delta$ of this ideal has facets $\{x_1,y_2\},\{x_2,y_1\},\{x_2,y_2\}$.  If $x_1<y_1$ then $\Delta$ is not shifted since $\{y_1,y_2\}\not\in \Delta$ and if $y_1<x_1$ then $\Delta$ is not shifted since $\{x_2,x_1\}\not\in \Delta$.
\end{example}

The structure of the induction in the argument that follows is the same as that given by Migliore and Nagel \cite{MN02} for strongly stable monomial ideals. The precise details of our argument and the connection to polarization will allow us to give a topological application in \cref{thm:ArtinianVertexDecomp}.

\begin{theorem}\label{thm:StableVertexDecomp}
    If $I$ is a stable monomial ideal that defines a Cohen--Macaulay quotient, its polarization $\P(I)$ is a glicci ideal whose Stanley--Reisner complex is vertex decomposable. 
\end{theorem}

\begin{proof}
Nagel and R\"omer \cite[Theorem 3.3]{NR08} showed that the Stanley--Reisner ideal of a vertex decomposable simplicial complex is glicci via a sequence of monomial basic double G-links of shift $1$. Hence, it suffices to show that $\P(I)$ is vertex decomposable.

Suppose that $I$ is a stable monomial ideal of height $\hgt(I)=c$ in $R$, a polynomial ring of dimension $n$. It is a consequence of the Eliahou--Kervaire \cite{EK} formula for projective dimension (which states $\pd(R/I)=\max\{m(u)\mid u\in G(I)\}$) and the fact that $I$ is a Cohen--Macaulay ideal (which implies $\pd(R/I) = \hgt(I)$ by the Auslander--Buchsbaum formula) that the minimal generators of $I$ belong to $\kappa[x_1,\ldots,x_c]$; thus, we may reduce to the case where $R/I$ is artinian and hence $n=c$. We proceed by induction on $c$ and on the greatest value $t$ so that $x_1^t \in I$, with the case when $c=1$ being trivial. 

There is a decomposition 
\begin{equation}\label{eq:decomp}
I=x_1I'+I_0
\end{equation}
 where 
 $I'=I:x_1$ and $I_0 \subset \kappa[x_2,\ldots,x_c]$ are monomial ideals uniquely determined by \cref{lem:formOfBDL}. Recall $I_0\subset I'$. If $\max\{t \colon x_1^t \in I\} = 1$, then $I' = (1)$ and $I = I_0+(x_1)$. By induction on $c$, $\P(I_0)$ is the Stanley--Reisner ideal of a vertex decomposable simplicial complex $\Delta$ on the ambient vertex set $\{2,\ldots,c\}$. But then $\Delta$ is also the Stanley--Reisner complex of $I$, viewed on the ambient vertex set $\{1,\ldots,c\}$. Hence $I$ is also the Stanley--Reisner ideal of a vertex decomposable simplicial complex. Henceforth, assume $\max\{t \colon x_1^t \in I\}>1$, in which case $I'$ is a proper ideal.

 Since $S/I$ is assumed artinian, $I_0$ contains pure powers of $x_i$ for every $i\geq 2$.  Thus $\hgt(I_0)=c-1$ and $S/I_0$ is Cohen--Macaulay. Since $S/I$ is assumed artinian and $I\subset I'$, $S/I'$ is artinian hence Cohen--Macaulay as well.
 
 We claim that the ideals $I'$ and $I_0$ are stable, the latter being viewed as an ideal of $\kappa[x_2,\ldots,x_c]$.
If $u$ is a monomial in $I'$ and $i<m(u)$ (in particular $1<m(u)$), then $x_1u\in I$ and therefore $x_i(x_1u)/x_{m(u)}=x_1(x_iu/x_{m(u)})\in I$ because $I$ is stable. Hence $x_iu/x_{m(u)}\in I'$, whence $I'$ is stable. Similarly, stability of $I_0$ follows directly from stability of $I$.
 
Transforming the decomposition \cref{eq:decomp} by polarization we obtain
\begin{equation}\label{eq:SSVertexDecomp}
\P(I)=x_{1,1}\B+\P(I_0), \, \qwhere \, \B = \P(x_1I'):x_{1,1}.
\end{equation}

Because $I' = I:x_1$, $\max\{t \colon x_1^t \in I'\} = \max\{t \colon x_1^t \in I\}-1$ and so, by induction, $\P(I')$ is the Stanley--Reisner ideal of a vertex decomposable simplicial complex. Hence, by \cref{lem:polarization-in-disguise}, so too is $\B$. By induction on $c$, $\P(I_0)$ is also the Stanley--Reisner ideal of a vertex decomposable simplicial complex. 

Let $\Delta$ denote the Stanley--Reisner complex of $\P(I)$, and let $v$ denote the vertex of $\Delta$ corresponding to the variable $x_{1,1}$. Note that $\B$ is the Stanley--Reisner ideal of the cone over $\lk_{\Delta}(v)$ with apex $v$. Similarly, $\P(I_0)$ is the Stanley--Reisner ideal of the cone over $\del_{\Delta}(v)$ with apex $v$. Because a simplicial complex is vertex decomposable if and only if any cone over it is \cite[Proposition 2.4]{PB80}, $\lk_\Delta(v)$ and $\del_\Delta(v)$ are vertex decomposable; therefore, $\Delta$ is vertex decomposable, as desired.
 \end{proof}

In the proof of the previous theorem, an important reduction occurred to the artinian case, to which we now turn our attention.  Huneke and Ulrich \cite{HU07} previously gave an algorithm for determining if an artinian ideal is in the CI-liaison class of a complete intersection and, in the same paper, showed that all artinian monomial ideals are glicci if non-homogeneous links are permitted, an allowance that does not give rise to G-links of subschemes of projective space. We consider polarizations of artinian monomial ideals, with motivation coming from recent work on the topology of the Stanley--Reisner complexes of such ideals.  Murai \cite[Remark 1.8]{M2011} first showed that the Stanley--Reisner complex of the polarization of an artinian monomial ideal is vertex decomposable.  See also \cite{FM22, AFL22, HLSR22} for recent generalizations to non-standard notions of polarization.  

\begin{theorem}[See also \cite{M2011} for vertex decomposability]\label{thm:ArtinianVertexDecomp}
    If $I$ is an artinian monomial ideal, its polarization $\P(I)$ is a glicci monomial ideal that is the Stanley--Reisner ideal of a vertex decomposable simplicial complex. 
\end{theorem}
\begin{proof}
    The proof follows the same path as that of \cref{thm:StableVertexDecomp}. The key observation needed for the induction is that the ideals $I'$ and $I_0$ in \cref{eq:decomp} give artinian quotients $S/I'$ and $\kappa[x_2,\ldots, x_n]/I_0$, respectively. This shows that the simplicial complex associated to $\P(I)$ is vertex decomposable.
    \end{proof}

\begin{remark}\label{r:grafted} 
The polarization of an artinian monomial ideal is known to form the facet ideal of what is called a {\em grafted} complex (see~\cite{Fa2005}  and~\cite{CFHNVT2025}). This class includes the edge ideals of {\em whiskered} graphs (see~\cite{V1990}).  Combining these results with  \cref{thm:ArtinianVertexDecomp} shows that facet ideals of grafted complexes - and in particular edge ideals of  whiskered graphs - are glicci.  
\end{remark}

\section{Homogeneous ideals and Gorenstein liaison: background}\label{s:geometric-vertex-decomposition}
\subsection{Geometric vertex decomposition}\label{ss:gvd-background}

The remainder of this paper makes consistent use of term orders, initial ideals, and Gr\"obner bases.  For a review of standard facts and terminology, we refer the reader to \cite[Chapter 15]{Eis95}.

Knutson, Miller, and Yong \cite{KMY09} introduced geometric vertex decomposition in their study of vexillary matrix Schubert varieties, otherwise known as one-sided mixed ladder determinantal ideals.

\begin{defn}\label{def:ycomp}
Suppose that $R = \kappa[x_1, \ldots, x_n]$ is equipped with a term order $<$.  Fix $j \in [n]$, and set $y = x_j$. For a polynomial $f = \sum_{i=0}^t \alpha_i y^i \in R$, where $\alpha_t \neq 0$ and no term of any $\alpha_i$ is divisible by $y$, write $\init_{y} (f) = \alpha_t y^t$.  If $\init_<(f) = \init_<(\init_{y}(f))$ for all $f \in R$, we say that $<$ is a \textbf{$y$-compatible term order}.  

For an ideal $I$ of $R$, write $\init_{y}(I) = (\init_{y}(f) \colon f \in I)$.  If there exists a $<$-Gr\"obner basis $\G$ of $I$ such that $\init_<(g) = \init_<(\init_{y}(g))$ for all $g \in \G$, we call $<$ \textbf{$y$-compatible with respect to $I$}.  
\end{defn}

If $<$ is $y$-compatible with respect to $I$, then it follows from \cref{def:ycomp} that $\init_<(I) = \init_<(\init_{y}(I))$.

We repeat some basic facts about geometric vertex decomposition, which can be found in \cite[Section 2.1]{KMY09}.  For an ideal $I$ and term order $<$ that is $y$-compatible with respect to $I$, define 
\begin{eqnarray*}
    C_{y,I} &=& \bigcup_{i \geq 1} (\init_{y} (I): y^i) \\
    N_{y,I} &=& (I \cap \kappa[x_1, \ldots, \hat{y}, \ldots , x_n])R.
\end{eqnarray*}
 Note that $N_{y,I}+(y) = \init_{y}(I)+(y)$ and that $C_{y,I}$ and $N_{y,I}$ each have a generating set that does not involve $y$.  If $\G = \{yq_1+r_1, \dots, yq_k+r_k, h_1, \ldots, h_\ell\}$, where no term of any $q_i$, $r_i$, or $h_i$ is divisible by $y$, then \[
 \init_{y}(I) = (yq_1, \dots, yq_k, h_1, \ldots, h_\ell), \quad C_{y,I} = (q_1, \dots, q_k, h_1, \ldots, h_\ell), \quad N_{y,I} = (h_1, \ldots, h_\ell),
 \] and \[
\init_{y}(I) = yC_{y,I}+N_{y,I} = C_{y,I} \cap (N_{y,I}+(y)).
\] The given generating sets of $\init_{y}(I)$, $C_{y,I}$, and $N_{y,I}$ are $<$-Gr\"obner bases for them, respectively.
\begin{defn}
    In the setup described above, we call the equation $\init_{y}(I) = C_{y,I} \cap (N_{y,I}+(y))$ a \textbf{geometric vertex decomposition of $I$ with respect to $y$.}  We call a geometric vertex decomposition \textbf{nondegenerate} if $C_{y,I} \neq (1)$ and $\sqrt{C_{y,I}} \neq \sqrt{N_{y,I}}$.
\end{defn}

If $R/I$ is equidimensional and $\init_{y}(I) = C_{y,I} \cap (N_{y,I}+(y))$ is a nondegenerate geometric vertex decomposition, then $\hgt(C_{y,I}) = \hgt(N_{y,I})+1 = \hgt(I)$ \cite[Lemma 2.8]{KR21}. For further discussion of degenerate and nondegenerate geometric vertex decompositions, see \cite[Section 2]{KR21}.

\begin{example}\label{ex:gvd}
    Let $I = (x_{11}x_{22}-x_{12}x_{21}, x_{11}x_{23}-x_{13}x_{21}, x_{12}x_{23}-x_{22}x_{13})$, i.e., the ideal generated by the size $2$ minors of a $2 \times 3$ generic matrix. Note that the given generating set is a universal Gr\"obner basis and so, in particular, a Gr\"obner basis with respect to any $x_{11}$-compatible term order. Then $C_{x_{11},I} = (x_{22},x_{23})$, $N_{x_{11},I} = (x_{12}x_{23}-x_{22}x_{13})$, and \[
\init_{x_{11}}(I) = (x_{11}x_{22}, x_{11}x_{23}, x_{12}x_{23}-x_{22}x_{13}) = (x_{22},x_{23}) \cap (x_{12}x_{23}-x_{22}x_{13}, x_{11})
    \] is a geometric vertex decomposition of $I$ with respect to $x_{11}$.
\end{example}

If $I_\Delta$ is a squarefree monomial ideal and $v$ is the vertex of $\Delta$ corresponding to $y$, then $C_{y,I_\Delta} = I_{\lk_{\Delta} (v)}$, and $N_{y,I_\Delta} = I_{\del_{\Delta} (v)}$.  For an arbitrary ideal $I$ that admits a geometric vertex decomposition $\init_{y}(I) = C_{y,I} \cap (N_{y,I}+(y))$, we will call $C_{y, I}$ the \textbf{geometric link} and $N_{y, I}$ the \textbf{geometric deletion} of $I$ at $y$.
 
\begin{defn}\label{def:gvd}
An ideal $I\subseteq R = \kappa[x_1, \ldots, x_n]$ is \textbf{geometrically vertex decomposable} if $I$ is unmixed and if
\begin{enumerate}
\item $I = ( 1)$ or $I$ is generated by a (possibly empty) set of indeterminates in $R$, or
\item for some variable $y = x_j$ of $R$, 
$\init_y I =  C_{y,I}  \cap (N_{y,I}+ (y))$ is a geometric vertex decomposition and the contractions of $N_{y,I}$ and $C_{y,I}$ to $\kappa[x_1,\dots,\widehat{y},\dots, x_n]$ are geometrically vertex decomposable.
\end{enumerate}
\end{defn}

An ideal $I\subseteq R = \kappa[x_1, \ldots, x_n]$ is \textbf{weakly geometrically vertex decomposable} if $I$ is unmixed and if
\begin{enumerate}
\item $I = ( 1)$ or $I$ is generated by a (possibly empty) set of indeterminates in $R$, or
\item for some variable $y = x_j$ of $R$, 
$\init_y I = C_{y,I} \cap (N_{y,I}+(y))$ is a geometric vertex decomposition, the contraction of $C_{y,I}$ to the ring $\kappa[x_1,\dots,\widehat{y},\dots, x_n]$ is weakly geometrically vertex decomposable, and $N_{y,I}$ is a radical and Cohen--Macaulay ideal.
\end{enumerate}

If $I$ is weakly geometrically vertex decomposable, then $I$ is radical and $R/I$ is Cohen--Macaulay \cite[Corollary 4.8]{KR21}.  Consequently, and as the name suggests, a geometrically vertex decomposable ideal is weakly geometrically vertex decomposable. If $I_\Delta$ is a squarefree monomial ideal, then $I$ is (weakly) geometrically vertex decomposable if and only if $\Delta$ is (weakly) vertex decomposable. 

\subsection{Geometric vertex decomposition and elementary G-biliaison}\label{ss:gvd-and-elementary-G-biliaison}
Geometric vertex decomposition is intimately related to elementary G-biliaison, which we review now.  The reader may find the definition of elementary G-biliaison reminiscent of the definition of basic double G-linkage. We will describe their relationship momentarily.

 \begin{defn}\label{def:Gbiliaison}
Let $I$ and $C$ be homogeneous, saturated, unmixed ideals of $R = \kappa[x_1, \ldots, x_n]$.  Suppose there exists a homogeneous ideal $N$ satisfying all of the following properties:
\begin{itemize}
    \item $N \subseteq I \cap C$ ;
    \item $R/N$ is Cohen--Macaulay and $G_0$;
    \item $\hgt(C) = \hgt(N)+1=\hgt(I)$;
    \item there is an isomorphism of graded $R/N$-modules $I/N \cong [C/N](-\ell)$ for some $\ell\in \ZZ$.
\end{itemize}
Then we say that $I$ is obtained from $C$ by an \textbf{elementary G-biliaison of shift $\ell$}.  
\end{defn}

\begin{theorem}\cite[Theorem 3.5]{Har07}
Let $I$ and $C$ be homogeneous, saturated, unmixed ideals of a polynomial ring over a field.  If $I$ is obtained from $C$ by an elementary G-biliaison, then $I$ is G-linked to $C$ in two steps.
\end{theorem}

From a basic double G-link $D = A+fB$ with $d = \deg(f)$, one may construct the elementary G-biliaison $[B/A](-d) \cong D/A$, where the map is multiplication by $f$.  Also, though, from an elementary G-biliaison $I/N \cong [C/N](-\ell)$, one may construct a pair of basic double G-links that together link $I$ and $C$ \cite[Remark 1.13(3)]{GMN13}.  Consequently, basic double G-links and elementary G-biliaisons generate the same equivalence class.

\begin{theorem}\cite[Theorem 4.4]{KR21}
    If $I$ is homogeneous and weakly geometrically vertex decomposable, then $I$ is G-linked to a complete intersection generated by linear forms via a sequence of elementary G-biliaisons of shift $1$.   
\end{theorem} 

\begin{example}
    Again with $I = (x_{11}x_{22}-x_{12}x_{21}, x_{11}x_{23}-x_{13}x_{21}, x_{12}x_{23}-x_{22}x_{13})$, as in \cref{ex:gvd}, the corresponding elementary G-biliaison is \[
    I/(x_{12}x_{23}-x_{22}x_{13}) \cong [(x_{22},x_{23})/(x_{12}x_{23}-x_{22}x_{13})](-1),
    \] where the isomorphism is given by multiplication by $x_{22}/(x_{11}x_{22}-x_{12}x_{21})$.  The isomorphism may also, for example, be described as multiplication by $x_{23}/(x_{11}x_{23}-x_{13}x_{21})$, as the two fractions are equivalent modulo $(x_{12}x_{23}-x_{22}x_{13})$.

    This elementary G-biliaison is induced from the geometric vertex decomposition described in \cref{ex:gvd}.
\end{example}

\section{Geometric polarization and Gorenstein liaison}
\label{sect:geoPol}

\subsection{Definitions and basic properties}

In this subsection, we give a definition of polarizations for Gr\"obner bases of arbitrary ideals in a polynomial ring and prove basic results in analogy with polarization of a monomial ideal.  With an eye towards passing information about elementary G-biliaisons between an ideal and a geometric polarization of it (i.e., an ideal generated by the polarization of one of its Gr\"obner bases), our primary interest will be in deciding when the polarization of a Gr\"obner basis of a homogeneous ideal is again a Gr\"obner basis (\cref{thm:inducedGB}).  We begin by defining terms and giving examples and then show (in \cref{prop:redSuffices}) that the reduced Gr\"obner basis of an ideal with respect to a fixed term order is the appropriate object of study in this context.

\begin{defn}\label{def:geoPol}
Fix a variable $y = x_j$ of $R = \kappa[x_1, \ldots, x_n]$.  For $g \in R$, write $g=\sum_{i=0}^t y^ir_i$, $r_i \in R$, where $r_t \neq 0$ and $y$ does not divide any term of any $r_i$.  Using a new variable $y'$, define $$\P_y(g) = r_0+yr_1+\sum_{i=2}^t y(y')^{i-1}r_i \in R[y'].$$ For an ideal $I$, term order $<$, and $<$-Gr\"obner basis $\G$, define $\P_y(\G)=\{\P_y(g) \colon g \in \G\}$.  We call $\P_y(\G)$ the \bf{one-step geometric polarization of $\G$ with respect to $y$}. 
\end{defn}

Note that $(\G, y-y') = (\P_y(\G), y-y')$.  Note also that, if $I$ is a monomial ideal and $\G$ is its set of minimal monomial generators, a sequence of one-step polarizations ultimately yields the minimal monomial generators for the polarization of $I$ (in the usual sense, i.e., that of \cref{def:polarization}).  

For $0 \neq f \in R$, let $\deg_y(f)$ denote the greatest power of $y$ that divides at least one term of $f$. In particular, $\deg_y(f) = 0$ if no term of $f$ is divisible by $y$.  Define $\deg_y(0) = -\infty$.

\begin{notation}\label{not:indOrder} Let $R = \kappa[x_1, \ldots, x_n]$, fix a variable $y$ of $R$, let $y'$ denote a new variable, and let $R' = R[y']$.  Fix a term order $<$ on $R$.  If $\mu$ is a monomial in $R'$, let $\depol(\mu)$ denote the monomial of $R$ obtained by replacing in $\mu$ each $y'$ with a $y$. Define a term order $\prec$ on $R'$ in the following way: $\mu \prec \nu$ if $\depol(\mu)<\depol(\nu)$ or if $\depol(\mu)=\depol(\nu)$ and $\deg_{y}(\mu)<\deg_y(\nu)$.
\end{notation}

Polarization in the polynomial setting depends on a choice of Gr\"obner basis, not merely on an ideal.

\begin{example}\label{ex:redundant-gens-polarization-not-GB}
If $\G_1 = \{yt-z^2\}$ and $\G_2 = \{yt-z^2, y^2t-yz^2\}$, then $\P_{y}(\G_1) = \{yt-z^2\}$ while $\P_{y}(\G_2) = \{yt-z^2,yy't-yz^2\}$.  Although the second element of $\G_2$ is merely a multiple of the first, $(\P_{y}(\G_2)) \not\subseteq (\P_{y}(\G_1))$, and $(\P_{y}(\G_2))$ is not a principal ideal.
\end{example}

\begin{example}\label{ex:termOrderDependent}
    Consider $I = (y^2x,y^2z,yw-yz+t^2) \subset \QQ[t,w,x,y,z]$.  If $<_1$ is the lexicographic term order determined by $t<z<w<x<y$, then the reduced $<_1$-Gr\"obner basis of $I$ is \[
    \G_1 = \{xt^4+zt^4, yw-yz+t^2, yxt^2+yzt^2, y^2x+y^2z\},
    \] and \[
    \P_y(\G_1) = \{xt^4+zt^4, yw-yz+t^2, yxt^2+yzt^2, yy'x+yy'z\}.
    \] Meanwhile, if $<_2$ is the lexicographic order determined by $t<x<w<z<y$, then the $<_2$-Gr\"obner basis of $I$ is 
    \[
    \G_2 = \{zt^4+xt^4, ywt^2+yxt^2+t^4, yz-yw-t^2, y^2w+y^2x+yt^2\},
    \] and \[
    \P_y(\G_2) = \{zt^4+xt^4, ywt^2+yxt^2+t4 yz-yw-t^2, yy'w+yy'x+yt^2\}.
    \] Then $(\P_y(\G_1)) \neq (\P_y(\G_2))$.  For example, $y'wt^2+y'xt^2+t^4 \in (\P_y(\G_2)) \setminus (\P_y(\G_1))$. 
    
    In this example, neither $\P_y(\G_1)$ nor $\P_y(\G_2)$ forms a Gr\"obner basis under the induced term order of \cref{not:indOrder}.  Note also that both $<_1$ and $<_2$ are $y$-compatible term orders, which one would expect to behave relatively well with respect to polarization at $y$.  

        Computations for this example were performed in Macaulay2 \cite{M2}.
\end{example}

The above examples show that our definition of polarization requires a choice of term order and Gr\"obner basis.  For application to elementary G-biliaison, we will need to focus on Gr\"obner bases whose polarizations are again Gr\"obner bases. The following proposition shows that it suffices to restrict our consideration of polarizations to the (finitely many) reduced Gr\"obner bases of an ideal.

\begin{proposition}
    \label{prop:redSuffices}
    Let $I$ be an ideal of $R=\kappa[x_1, \ldots, x_n]$ and  $y$ a variable in $R$.  
    Suppose that $<$ is a $y$-compatible term order, that $\G$ is a $<$-Gr\"obner basis of $I$, and that $\G_{red}$ is the reduced $<$-Gr\"obner basis of $I$.  If $\P_y(\G)$ is a Gr\"obner basis under the induced order $\prec$ of \cref{not:indOrder}, then $\P_y(\G_{red})$ is a reduced $\prec$-Gr\"obner basis.
\end{proposition}
\begin{proof}
    Suppose there exists some $a \in \G_{red}$ and $\tilde{a} \in \G$ so that the leading term $\mu$ of $a$ divides some non-leading term $\nu$ of $\tilde{a}$. Assume that $a$ and $\tilde{a}$ have been chosen to maximize $\nu$ (with respect to $<$). 
    
    Set $\gamma = \nu/\mu$, and let $\widetilde{\G}$ be the set obtained from $\G$ by replacing $\tilde{a}$ by $\tilde{a}-\gamma a$.  Then $(\init_< (g) \colon g \in \G) = (\init_< (g) \colon g \in \widetilde{\G})$, which is to say that $\widetilde{\G}$ is again a Gr\"obner basis of $I$.  Moreover, $(\init_\prec(\P_y(g)): g \in \G) = (\init_\prec(\P_y(g)) \colon g \in \widetilde{\G})$ and $\P_y(\widetilde{\G}) = \P_y(\G)$, which is to say that $\P_y(\widetilde{\G})$ is again a $\prec$-Gr\"obner basis.  
    
    Now $\kappa$-multiples of $\nu$ occur strictly fewer times among the non-leading terms of elements of $\widetilde{\G}$ than they do among the non-leading terms of $\G$. If no $\kappa$-multiple of $\nu$ occurs among the non-leading terms of elements $\widetilde{\G}$, then the greatest among non-leading terms of $\widetilde{\G}$ divisible by a leading term of an element of $\G$ is less than $\nu$.  By Gr\"obner induction, we may now assume that no non-leading term of any element of $\G$ is divisible by the leading term of any element of $\G_{red}$, which forces $\G_{red} \subseteq \G$.

    We proceed by induction on $|\G|-|\G_{red}|$.  The case $|\G|-|\G_{red}|=0$ is trivial.  If $|\G|-|\G_{red}|>0$, choose $f \in \G \setminus \G_{red}$, and set $\G' = \G \setminus \{f\}$.  That is, $\G=\G' \cup \{f\}$.  
    
  Suppose first that the leading term of $f$ is divisible by the leading term of some $h \in \G$ that does not involve $y$.  Then $\P_y(h) = h$ and $\init_<(h) = \init_\prec(h)$, and so $\init_\prec(f) \in (\init_\prec(g) \colon g \in \P_y(\G))$.  Hence, $(\init_\prec(g) \colon g \in \P_y(\G')) = (\init_\prec(g) \colon g \in \P_y(\G))$.
    
    Alternatively, if the leading term of $f$ is not divisible by the leading term of some $h \in \G$ that does not involve $y$, then, because $\G'$ is a Gr\"obner basis, the leading term of $f$ must be divisible by the leading term of some element of the form $yd+r \in \G'$, where no term of $r$ is divisible by $y$.  Write $f = \alpha(yd+r)+\beta$ where $\init_<(f) = y \init_<(\alpha)\init_<(d)$. Then \[
    \init_{\prec}(\P_y(f)) \in (\init_{\prec}(\P_y(yd))) = (\init_{\prec}(\P_y(yd+r))) \subseteq (\init_{\prec}(g) \colon g \in \P_y(\G')).
    \] Thus, $(\init_\prec(g) \colon g \in \P_y(\G')) = (\init_\prec(g) \colon g \in \P_y(\G))$, and so $\P_y(\G')$ is a Gr\"obner basis because $\P_y(\G)$ is.  Because $|\G'|-|\G_{red}| < |\G|-|\G_{red}|$, this completes the proof  of the claim that $\P_y(\G_{red})$ is a Gr\"obner basis.

To see that $\P_y(\G_{red})$ is reduced, suppose there exist distinct elements $g, g' \in \P_y(\G_{red})$ so that $\mu = \init_\prec(g)$ divides some term $\nu$ of $g'$.  Then $\depol(\mu) = \init_<(\depol(g))$ divides $\depol(\nu)$, which is a term of $\depol(g')$.  Because  $\depol(g)$ and $\depol(g')$ are, by the construction of $\P_y(\G_{red})$, distinct elements of $\G_{red}$, this relationship contradicts the assumption that $\G_{red}$ is the reduced $<$-Gr\"obner basis of $I$.
\end{proof}

\begin{remark}\label{rem:polarization-generates-same-ideal-as-reducedGBpolarization}
    Under the hypotheses and with the notation of \cref{prop:redSuffices}, we see from the proof that $(\P_y(\G)) = (\P_y(\hat{\G}))$ for some $<$-Gr\"obner basis $\hat{\G} \supseteq \G_{red}$. It follows that $(\P_y(\G)) \supseteq  (\P_y(\G_{red}))$. We do not yet have enough information to argue that indeed $(\P_y(\G)) = (\P_y(\G_{red}))$, though, after proving \cref{thm:inducedGB}, we will.  We saw in \cref{ex:redundant-gens-polarization-not-GB} that, without the assumption that $\P_y(\G)$ forms a Gr\"obner basis, it may be that $(\P_y(\G)) \neq (\P_y(\G_{red}))$, even if $\P_y(\G_{red})$ does form a Gr\"obner basis.
\end{remark}

Suppose that $I$ is a monomial ideal and that $\G = G(I)$ is its minimal monomial generating set. Much of the theory of polarization in this traditional monomial setting rests on the fact that $y-y'$ is not a zerodivisor on $R/I$ or on $R'/(\P_y(\G))$. We will see that the same is true in the setting of geometric polarization.  The next theorem shows that $y-y'$ not a zerodivisor on $R/I$ or on $R'/(\P_y(\G))$ is precisely the condition under which the polarization of a Gr\"obner basis is again a Gr\"obner basis, under the induced order described in \cref{not:indOrder}.

As usual, let $R = \kappa[x_1, \ldots, x_n]$, equipped with the standard grading. Let $\HS_M(t)$ denote the Hilbert series of the graded $R$-module $M$.

\begin{theorem}\label{thm:inducedGB}
    Fix a homogeneous ideal $I$ of $R = \kappa[x_1, \ldots, x_n]$, a variable $y$ of $R$, a new variable $y'$, a $y$-compatible term order $<$, and a $<$-Gr\"obner basis $\G$ of $I$.  Then $\P_y(\G)$ is a $\prec$-Gr\"obner basis of $(\P_y(\G))$ if and only if $y-y'$ is a nonzerodivisor on $R[y']/(\P_y(\G))$.

    Furthermore, if $\G$ is a reduced Gr\"obner basis and $\P_y(\G)$ a Gr\"obner basis, then it is reduced as well.
\end{theorem} 
\begin{proof}
Set $R' = R[y']$.   Write \[
\G = \{yd_1+r_1, \ldots, yd_k+r_k, h_1, \ldots, h_\ell\},
    \] where $y$ does not divide any term of any $r_j$ or any $h_j$.  For each $i \in [k]$, let $d_i'$ be the polynomial obtained from $d_i$ by writing a $y'$ in place of each $y$.  Set \[
I' = (\P_y(\G)) = (yd'_1+r_1, \ldots, yd'_k+r_k, h_1, \ldots, h_\ell).
    \] The natural map given by $y' \mapsto y$ gives a graded isomorphism $R'/(I'+(y-y')) \cong R/I$.  Let \[
J = \init_<(I) = (\init_<(yd_1), \ldots, \init_<(yd_k), \init_<(h_1), \ldots, \init_<(h_\ell))
\] and \[
J' = (\init_{\prec}(yd'_1), \ldots, \init_{\prec}(yd'_k), \init_{\prec}(h_1), \ldots, \init_{\prec}(h_\ell)).
\] Note that $J'$ is a partial polarization of $J$.  

 Clearly, $J' \subseteq \init_{\prec}(I')$. Equality occurs if and only if $\HS_{R'/J'}(t)=\HS_{R'/\init_{\prec}(I')}(t)$.  We now compute each of these Hilbert series.  Equalities $\HS_{R'/J'}(t)=\HS_{R'/JR'}(t)=\HS_{R'/IR'}(t)$ hold since $J'$ is a partial polarization of $J$ and $J=\init_<(I)$. Moreover the decomposition $R'/IR'=R/I\otimes_\kappa \kappa[y']$ yields $\HS_{R'/IR'}(t)=\HS_{R/I}(t)\cdot \HS_{\kappa[y']}(t)=\HS_{R/I}(t)/(1-t)$. Putting this together we have
\begin{equation}\label{eq:5.1}
    \HS_{R'/J'}(t)= \HS_{R'/\init_{\prec}(I')}(t)  
    \iff \frac{\HS_{R/I}(t)}{1-t}=\HS_{R'/I'}(t)  
    \iff  \HS_{R/I}(t)=(1-t)\HS_{R'/I'}(t).
\end{equation}
Restricting the exact sequence \[
0 \rightarrow \ker \rightarrow R'/I' (-1) \xrightarrow{y-y'} R'/I' \rightarrow R'/(I'+(y-y')) \rightarrow 0,
\] where $\ker$ denotes the kernel of the multiplication by $y-y'$ map, to each degree $t$ gives rise to the exact sequence of finite-dimensional vector spaces 
\[
0 \rightarrow [\ker]_t \rightarrow [R'/I' (-1)]_t \xrightarrow{y-y'} [R'/I']_t \rightarrow [R'/(I'+(y-y'))]_t \rightarrow 0.
\]
It follows that \[
\HS_{\ker}(t)+\HS_{R'/I'}(t)=\HS_{R'/I'(-1)}(t)+\HS_{R'/(I'+(y-y'))}(t).
\]
Using the identity $\HS_{R'/I'(-1)} = t\HS_{R'/I'}$ and the isomorphism $R'/(I'+(y-y')) \cong R/I$, we combine like terms to obtain

\begin{equation}\label{eq:5.2}
\HS_{\ker}(t) = \HS_{R/I}(t)-(1-t)\HS_{R'/I'}(t).
\end{equation}

Combining \cref{eq:5.1} and  \cref{eq:5.2} we see that $\HS_{R'/J'}(t)=\HS_{R'/\init_{\prec}(I')}(t)$ if and only if $\HS_{\ker}(t)=0$ if and only if $y-y'$ is a nonzerodivisor on $R'/I'$.

This final sentence of the theorem follows by the same argument given for reducedness in \cref{prop:redSuffices}.
\end{proof}

\begin{example}\label{ex:badGens}
    If $I = (y^2-xz, yr-st)$, then the reduced Gr\"obner basis of $I$ with respect to any $y$-compatible term order is \[
\G = \{y^2-xz,yr-st,yst-rxz,xzr^2-s^2t^2\}, 
    \] in which case \[
\P_y(\G) = \{yy'-xz,yr-st,yst-rxz,xzr^2-s^2t^2\}. 
    \] 
We first claim that $y-y'$ is a zerodivisor on $R'/\P_y(\G)$.  Indeed, in $R'/(\P_y(\G))$ \[
st(y-y') = sty-sty' = rxz-ryy' = r(xz-yy')=0, 
    \] although $st \neq 0$ because $st$ is a degree $2$ form that is not in the $\kappa$-span of $\{yy'-xz,yr-st\}$.

We next claim that $\P_y(\G)$ is not a Gr\"obner basis under the induced term order of \cref{not:indOrder}, as predicted by \cref{thm:inducedGB}. Note that $y'st-rxz \in (\P_y(\G))$ although $y'st$ is not divisible by $yy'$,$yr$,$yst$, $xzr^2$, or $s^2t^2$. (Depending on the $y$-compatible term order $<$, either of $xzr^2$, or $s^2t^2$ may be the leading term of $xzr^2-s^2t^2$.)
\end{example}

\begin{remark}We return to the setting of \cref{prop:redSuffices} and \cref{rem:polarization-generates-same-ideal-as-reducedGBpolarization}.  By \cref{thm:inducedGB}, $y-y'$ is a nonzerodivisor on both $R'/\P_y(\G)$ and $R'/\P_y(\G_{red})$.  Having established that $(\P_y(\G_{red})) \subseteq (\P_y(\G))$ and noting that $(\P_y(\G_{red}), y-y') = (\P_y(\G), y-y')$, a routine Hilbert function computation shows that $(\P_y(\G_{red})) = (\P_y(\G))$.  One interpretation of this result is that (for a fixed choice of algebra generators of $R$) there are only finitely many ideals one might want to call a polarization of a given ideal $I$, each one obtained from a reduced Gr\"obner basis of $I$.
\end{remark}

The next proposition extends some core facts about polarization from the monomial setting to the geometric setting and also includes a note about primality not relevant in the monomial setting.

\begin{proposition}\label{prop:stFacts}
   Fix $R = \kappa[x_1, \ldots, x_n]$, a variable $y$ of $R$, and a new variable $y'$. Let $R' = R[y']$, where both $R$ and $R'$ are equipped with the standard grading. Suppose that $I$ is a homogeneous ideal, $<$ is a $y$-compatible term order, and $\G$ is a $<$-Gr\"obner basis of $I$. If $y-y'$ is not a zerodivisor on $R'/(\P_y(\G))$, then
\begin{enumerate}
    \item $\beta^R_{ij}(I) =\beta^{R'}_{ij} (\P_y(\G))$ for all i and j, where $\beta_{ij}$ denotes the $ij^{th}$ graded Betti number,
\item $H_{R/I}(t) = (1-t)H_{R'/(\P_y(\G))} (t)$,
\item $\hgt(I) = \hgt(\P_y(G))$,
\item $\pd_R(R/I) = \pd_{R'}(R'/(\P_y(\G)))$, where $\pd$ denotes projective dimension,
\item $\reg_R (R/I) = \reg_{R'}(R'/(\P_y(\G)))$, where $\reg$ denotes Castelnuovo--Mumford regularity,
\item $R/I$ is Cohen--Macaulay (respectively, Gorenstein) if and only if $R'/(\P_y(\G))$ is Cohen--Macaulay (respectively, Gorenstein), and 
\item if $I$ is a prime ideal of $R$, then $(\P_y(\G))$ is a prime ideal of $R'$.
\end{enumerate}
\end{proposition}
\begin{proof}

    Statement (1) follows from the well-known fact that if $M$ is an $R'$ module and $u\in R'$ is a regular element on both $R'$ and $M$, then there are degree-preserving isomorphisms $\Tor^{R'}_i(M,\kappa)\cong \Tor^{R'/(u)}_i(M/(u),\kappa)$ for all $i\geq 0$, see  \cite[Corollary 20.4]{Peeva}. We apply this for $u=y-y'$ and $M=R'/(\P_y(\G))$, obtaining that $ \beta^{R'}_{ij} (R'/(\P_y(\G)))=\beta^R_{ij}(R/I)$ upon identifying  $R'/(u)$ as $R$ and $R'/(\P_y(\G)), y-y')$ as $R/I$.

    One can apply \cite[Theorem 20.2]{Peeva} to obtain statement (2). Statement (3) follows from (2) as the Hilbert function determines the dimension of the respective rings. 
    Statements (4)--(6) follow from (1) as the Betti numbers determine the depth, projective dimension, regularity, and the last Betti number determines the Gorenstein property in the presence of Cohen-Macaulayness.

   (7)  Suppose that $I$ is a prime ideal of $R$.  Suppose for contradiction that there are homogeneous elements $r, s \in R' \setminus I'$ with $rs \in I'$.  Assume that $r$ and $s$ have been chosen so that $\deg(r)+\deg(s)$ is as small as possible. 
 Consider the map $f:R' \rightarrow R$ determined by $y' \mapsto y$, i.e., the depolarization map.  Then $f(I') = I$, and so $f(r)f(s) = f(rs) \in I$.  Because $I$ is prime, we may assume without loss of generality that $f(r) \in I$.  Then $r \in f^{-1}(I)= I'+(y-y')$.  Write $r = a+v(y-y')$ for homogeneous elements $v \in R'$ and $a \in I'$.  Then $as+v(y-y')s = (a+v(y-y'))s = rs \in I'$, and $as \in I'$, and so $v(y-y')s \in I'$.  By assumption, $y-y'$ is a nonzerodivisor on $R/I'$, and so $vs \in I'$.   But $\deg(v)+\deg(s) = \deg(r)+\deg(s)-1$, and so, but the assumption of minimality of degree and the supposition $s \notin I'$, we must have $v \in I'$.  But then, $r = a+v(y-y')$ with $a, v \in I'$, and so we conclude $r \in I'$.
\end{proof}
  
\begin{example}
The assumption that $y-y'$ is a nonzerodivisor on $R'/(\P_y(\G))$ is required for \cref{prop:stFacts}. Returning to \cref{ex:badGens}, $\hgt(I) = 2$, but $\hgt(\P_y(\G)) = 3$.

Separately, let $R = \QQ[y,x,z,r,s,t]$ equipped with lexicographic order $<$ extending $t<s<r<z<x<y$, and let $J = (y^2z+yt^2+s^3,ys+t^2)$.  Then the reduced $<$-Gr\"obner basis of $J$ is \[
\H = \{y^2z+yt^2+s^3, yzt^2-s^4+t^4,ys+t^2, zt^4+s^5-st^4\},
\] and \[
\P_y(\H) = \{yy'z+yt^2+s^3, yzt^2-s^4+t^4,ys+t^2, zt^4+s^5-st^4\}.
\] Now $J$ has height $2$ and defines a Cohen--Macaulay quotient while $(\P_y(\H))$ has height $3$ and does not. As required by \cref{prop:stFacts}, $y-y'$ is a zerodivisor on $R[y']/(\P_y(\H))$.

Computations in this example were performed in Macaulay2 \cite{M2}.
\end{example}

    One cannot recover in the geometric setting all of the standard facts about polarization in the monomial setting.  In particular, geometric polarization need not preserve containment, and the associated primes of $\P_y(\G)$ need not be the polarizations or relabelings of the associated primes of $(\G)$.

\begin{example}
  Consider $(y^2-z^2) = (y-z) \cap (y+z)$.  Then $(\P_y \{y^2-z^2\}) = (yy'-z^2)$, which is a prime ideal and which is not contained in either $(y-z)$ or $(y+z)$.  Note that this failure occurs even though $y-y'$ is a nonzerodivisor on $R'/(yy'-z^2)$.  This example also shows that the converse to \cref{prop:stFacts}(4) is false.
 \end{example} 

However, we can give a partial converse to \cref{prop:stFacts}(1):

\begin{lemma}
Let $R = \kappa[x_1, \ldots, x_n]$, fix a variable $y$ of $R$, let $y'$ denote a new variable, and let $R' = R[y']$, where both $R$ and $R'$ are equipped with the standard grading. Suppose that $I$ is a homogeneous ideal, $<$ is a $y$-compatible term order, and $\G$ is a $<$-Gr\"obner basis of $I$.  If $(\P_y(\G))$ is unmixed and $\hgt(I)=\hgt(\P_y(\G))$,  then $y-y'$ is not a zerodivisor on $R'/(\P_y(\G))$.
\end{lemma} 
\begin{proof}

    The identity
    $
    I+(y-y')=(\P_y(\G))+(y-y')
    $
    yields 
    \[
    \hgt\left((\P_y(\G)+(y-y')\right)=\hgt( I+(y-y'))=\hgt(I)+1>\hgt(I)=\hgt\left(\P_y(\G)\right).
    \]
Therefore $y-y'$ is not in any associated prime of $R'/(\P_y(\G))$ of height equal to the height of $(\P_y(\G))$. Since this ideal is unmixed, the conclusion follows.
\end{proof}

We end this subsection with a class of Gr\"obner bases whose polarizations are Gr\"obner bases under the induced term order of  \cref{not:indOrder}. If $G = \{g_i \colon i \in [r]\}$ is a set of polynomials in the ring $\kappa[x_1, \ldots, x_n]$ and $<$ a term order, recall that an expression $f = \sum_{i=1}^r f_i g_i+f'$ is called a \textbf{standard expression} for $f$ in terms of $G$ if no monomial of $f'$ is in the ideal $(\init_<(g_i) \colon i \in [r])$ and $\init_<(f) \geq \init_<(f_i g_i)$ for all $i \in [r]$. We call any such $f'$ a \textbf{remainder} of $f$ upon division by $G$.  For monic polynomials $g$ and $h$, the \textbf{$s$-polynomial} of $g$ and $h$ is defined to be \[
s(g,h) = \dfrac{\init_<(g)f}{\gcd(\init_<(f),\init_<(g))}-\dfrac{\init_<(f)g}{\gcd(\init_<(f),\init_<(g))}.
\]
Buchberger introduced Gr\"obner bases and, as a means to compute them algorithmically, $s$-polynomials in his 1965 PhD thesis \cite{Buc65}. He showed that, if the remainder of the $s$-polynomial of every pair of elements of a set $G$, upon division by $G$, is $0$, then $G$ is a Gr\"obner basis for the ideal it generates.

\begin{proposition}\label{prop:uniform-deg-in-y-polarizes}
Fix a homogeneous ideal $I$ of $R = \kappa[x_1, \ldots, x_n]$, a variable $y$ of $R$, a new variable $y'$, a $y$-compatible term order $<$, and a $<$-Gr\"obner basis $\G$ of $I$. Suppose there is some integer $t$ such that $\deg_y (g) \in \{0,t\}$ for each $g\in \G$. Then $\P_y (\G)$ is a Gr\"obner basis under the induced term order $\prec$ on $R' = R[y']$.
\end{proposition}

\begin{proof}
Let $g,h\in \G$. Without loss of generality, assume that $g$ and $h$ are monic.  We will consider the $s$-polynomial $s(\P_y(g),\P_y(h))$ and show that its remainder upon division by $\P_y (\G)$ is $0$.  To do this, we will first fix a standard expression \[
s(g,h) = c_1g_1+\cdots+c_rg_r 
\] for $s(g,h)$ in terms of $\G$ and show that \[
\P_y(s(g, h)) = 
s(\P_y(g),\P_y(h)) 
\qand 
\P_y(c_i g_i) = 
\P_y(c_i)\P_y(g_i) \qforall i \in [r],
\] from which it follows that \[ s(\P_y(g),\P_y(h))=\P_y(c_1)\P_y(g_1)+\cdots+\P_y(c_r)\P_y(g_r).
\]

In order to establish this claim, we will frequently make use of the observation that, if $u, v \in R$ and $\deg_y(u)=0$, then $\P_y(uv) = u \P_y(v) = \P_y(u)\P_y(v)$.  We consider three cases.

\noindent \textbf{Case 1:} Suppose first that $\deg_y(g) = \deg_y(h) = 0$. Then $\deg_y(c_ig_i) = 0$ for all $i \in [r]$, and so the claim follows from the above observation.

\noindent \textbf{Case 2:} Next suppose $\deg_y(g) = \deg_y(h) = t$. 
Then there are some monomials $m$ and $n$ such that 
$s(\P_y(g),\P_y(h)) = m\P_y(h) - n\P_y(g)$. 
By the assumption that $\deg_y(g) = \deg_y(h) = t$, we observe that neither $m$ nor $n$ is divisible by $y$ or $y'$. 
Consequently, 
\[
    s(\P_y(g),\P_y(h)) = m\P_y(h) - n\P_y(g)
     = \P_y(mh) - \P_y(ng)   = \P_y(s(g,h)).
\]
Now recall the standard expression
$s(g,h) = c_1g_1+\cdots + c_rg_r$ for $s(g,h)$ in terms of $\G$.
For each $i \in [r]$, if $g_i$ has a term divisible by $y$, then, by assumption, $\deg_y(g_i) = t$. Hence, because $\deg_y(c_ig_i) \leq \deg_y (s(g,h))\leq t$, we must have that $\deg_y(c_i) = 0$. 
Hence, for all $i \in [r]$, $\P_y(c_ig_i) = \P_y(c_i)\P_y(g_i)$ by our earlier observation. 

\noindent \textbf{Case 3:} Without loss of generality, assume $\deg_y(g) = t$ and $\deg_y(h) = 0$.
    Then there are monomials $m$ and $n$ such that 
$s(\P_y(g),\P_y(h)) = my(y')^{t-1}\P_y(h) - n\P_y(g)$,
where neither $m$ nor $n$ is divisible by $y$ or $y'$. Consequently, 
\[
    s(\P_y(g),\P_y(h)) = s(\P_y(g),h)= my'y^{t-1}h - n\P_y(g)
     = \P_y(my^t h) - \P_y(ng)   = \P_y(s(g,h)).
\]
The remainder of the claim in this third and final case follows via the argument from Case~2.

Having established that, for any $g, h \in \G$ and any standard expression \[
s(g,h) = c_1g_1+\cdots+c_rg_r,
\] we have \[
 s(\P_y(g),\P_y(h)) = \P_y(c_1)\P_y(g_1)+\cdots+\P_y(c_r)\P_y(g_r),
\] it remains only to show that \[
\init_\prec(s(\P_y(g),\P_y(h))) \geq \init_\prec(\P_y(c_i)\P_y(g_i)) \qforeach  i \in [r].
\] From the equality $s(\P_y(g),\P_y(h)) = \P_y(s(g, h))$, it follows that \[
\init_\prec(s(\P_y(g),\P_y(h))) = \init_\prec(\P_y(s(g, h))) = \P_y(\init_<(s(g,h))).
\] Similarly, from the equalities $\P_y(c_i)\P_y(g_i) = \P_y(c_ig_i)$ for each $i \in [r]$, it follows that \[
\init_\prec(\P_y(c_i)\P_y(g_i)) = \init_\prec(\P_y(c_ig_i)) = \P_y(\init_<(c_ig_i)).
\]
The assumption $\init_<(s(g,h)) \geq \init_<(c_ig_i)$ implies $\P_y(\init_<(s(g,h))) \geq \P_y(\init_<(c_ig_i))$.  Combining this information, we see $\init_\prec(s(\P_y(g),\P_y(h))) \geq \init_\prec(\P_y(c_i)\P_y(g_i))$ for all $i \in [r]$, which shows that $s(\P_y(g),\P_y(h))$ has remainder $0$ upon division by $\P_y(G)$.  Hence $\P_y(\G)$ forms a Gr\"obner basis.
\end{proof}

\begin{example}\label{ex:uniform-y-degree-polarizes}
    Consider the polynomial ring $\kappa[y,a,b,c,d]$ with the lexicographic monomial order determined by $y>a>b>c>d$ and the ideal $I = (y^2a+yd^2, y^2c+b^3, cd^2,b^3d^2,ab^3)$. Then the given generators form a Gr\"obner basis. Observe that each generator has $y$-degree $t=2$ or $0$. As required by \cref{prop:uniform-deg-in-y-polarizes}, the polarization of the generating set 
    \[
    \{y'ya+yd^2, y'yc+b^3, cd^2, b^3d^2, ab^3 \}
    \]
    forms a Gr\"obner basis of the lexicographic order with $y>y'>a>b>c>d$. 
\end{example}

\subsection{Geometric Polarization and Gorenstein Liaison }

In this subsection, we apply geometric polarization in the context of Gorenstein liaison.  Klein and Rajchgot \cite{KR21} gave a formula for producing an elementary G-biliaison from a geometric vertex decomposition of an ideal $I$ at a variable $y$, under suitable hypotheses.  A limitation of that result is that a geometric vertex decomposition requires $I$ to have a generating set that is linear in $y$.  The purpose of the next theorem is to relax that requirement by partially polarizing $I$ to produce an ideal that has a generating set that is linear in $y$, using the relationship between geometric vertex decomposition and G-biliaison to study the polarization, and then lifting the G-biliaison involving the polarization to a G-biliaison involving $I$.

\begin{theorem}\label{thm:polarLink}
    Let $R = \kappa[x_1, \ldots, x_n]$, and let $I$ be an unmixed and homogeneous proper ideal of $R$.  Fix $i \in [n]$, and write $y=x_i$.  Fix a $y$-compatible term order $<$ and a $<$-Gr\"obner basis \[
\G = \{yd_1+r_1, \ldots, yd_k+r_k, h_1, \ldots, h_\ell\}
    \] of $I$, where $y$ does not divide any term of any $r_j$ or any $h_j$.  Suppose that $y-y'$ is a nonzerodivisor on $R[y']/(\P_y(\G))$.  Set $D = (d_1, \ldots, d_k, h_1, \ldots, h_\ell)$ and $N = (h_1, \ldots, h_\ell)$, and let $D'$ be the ideal of $R[y']$ obtained from $D$ by the substitution $y \mapsto y'$.  Then the following are equivalent:
    \begin{enumerate}
        \item There is an elementary G-biliaison $(D/N)[-1] \cong I/N$, where the isomorphism is given by multiplication by any of the $yd_i+r_i/d_i$ for $i\in [k]$, all of which are equivalent modulo $N$;
        \item There is an elementary G-biliaison $(D'/NR[y'])[-1] \cong (\P_y(\G))/NR[y']$, where the isomorphism is given by multiplication by some $v'/d'$ with $v' \in (\P_y(\G))$ and $d' \in D'$ nonzerodivisors on $R[y']/NR[y']$ with $\init_y(v')/d' = y$;
        \item $(\P_y(\G))$ is unmixed and admits a nondegenerate geometric vertex decomposition at $y$, the geometric link is unmixed, and the geometric deletion is a Cohen--Macaulay and $G_0$ ideal.
    \end{enumerate}

    In this case,  be the fraction $yd_i+r_i/d_i$ is also represented modulo $N$ by $\depol(v')/\depol(d')$, in the notation of \cref{not:depolarization}.
    
\end{theorem}
\begin{proof}
Let $R' = R[y']$, and let $\prec$ be the term order of \cref{not:indOrder}.  For each $i \in [k]$, let $d_i'$ be the polynomial obtained from $d_i$ by the substitution $y \mapsto y'$.  By \cref{thm:inducedGB},    
 \[
\P_y(\G) = \{yd'_1+r_1, \ldots, yd'_k+r_k, h_1, \ldots, h_\ell\},
\] is a $\prec$-Gr\"obner basis of $I' = (\P_y(\G))$.  Note that $D' = (d_1', \ldots, d_k', h_1, \ldots, h_\ell)$.

  Then $\init_y(I') = D' \cap (NR'+(y)R')$ by \cite[Theorem 2.1]{KMY09}, and so $NR'$ is the geometric deletion and $D'$ is the geometric link of the geometric vertex decomposition of $I'$ at $y$. Because $(R/N)[y] \cong R'/NR'$, $R/N$ is Cohen--Macaulay and $G_0$ if and only if $R'/NR'$ is Cohen--Macaulay and $G_0$.  These equivalent conditions are therefore common to (1), (2), and (3), by the definition of elementary G-biliaison in the former two cases and by the text of the assumption in the latter.  Similarly, $D$ is umixed if and only if $D'$ is, and so that condition, too, is simultaneously obtained in (1), (2), and (3).

Thus, we may assume that $D$ and $D'$ are unmixed and that $N$ and $N'$ define Cohen--Macaulay and $G_0$ quotients.  

We will now show (3) $\iff$ (2).  The statement (3) $\implies$ (2) is proved in \cite[Corollary 4.3]{KR21}\footnote{In \cite{KR21}, the authors require $\kappa$ to be infinite.  They have since noticed that that assumption was unnecessary.  Rather than making a general choice of scalars, one may instead use the Prime Avoidance Lemma to find the nonzerodivisors required in the argument.}.  The converse (2) $\implies$ (3) is \cite[Theorem 6.1]{KR21}, up to the claim that the geometric vertex decomposition is nondegenerate.  Degeneracy in this case would mean $D' = (1)$ or $\sqrt{(D')} = \sqrt{NR[y']}$.  But these possibilities are ruled out by the requirement $\hgt(D') = \hgt(NR[y'])+1$, which is part of the definition of elementary G-biliaison.

To see (2) $\implies$ (1), we claim first that the map $v'/d'$ may be taken to be $yd_i'+r_i/d_i'$ for any $i \in [k]$. We will show that all such choices represent the same fraction modulo $N$.  Arguing from the assumption $\init_y(v')/d' = y$, one may show (for example, by using the construction in \cite{KR21}) that $v'$ may be taken to be $v' = yd'+r$ for some $r$ not involving $y$ and $d'$ may be assumed not to involve $y$. Because $(yd'+r)(d'_i)-d'(yd'_i+r_i) = rd_i'-d'r_i  \in I'$, $\P_y(\G)$ is a Gr\"obner basis, and $rd_i'-d'r_i$ does not involve $y$, we must have $rd_i'-d'r_i \in NR'$. 

From this viewpoint, it is clear that multiplication by $yd_i'+r_i'/d_i'$ maps each $d_i'+NR' \in D'/NR'$ to $yd_i'+r_i'+NR' \in I'/NR'$ and that multiplication by $d_i'/yd_i'+r_i'$ gives an inverse map.  Setting $d =\depol(d')$ and $v =\depol(v')$, each $rd_i'-d'r_i \in NR'$ implies each $rd_i-dr_i \in N$, and so multiplication by $v/d$ gives the desired elementary G-biliaison $(D/N)[-1] \cong I/N$.

The argument for (1) $\implies$ (2) is similar. If $r_id_j-r_jd_i \in N$ for all $i,j \in [k]$, then $r_id'_j-r_jd'_i \in NR[y']$ for all $i,j \in [k]$. Taking $v' = yd_1'+r_1$ and $d' = d_1'$, for example, gives the desired isomorphism, and clearly satisfies $\init_y(v')/d' = y$.
\end{proof}

\begin{example}\label{ex:mainTheroemApp1}
   Recall that the ideal $I = (yz-x^2,wz^2-y^2x,wxz-y^3)$ from \cref{ex:manySquares} has height $2$, defines a Cohen--Macaulay quotient, is minimally generated by $3$ elements, and is not weakly geometrically vertex decomposable.  The given generators form a Gr\"obner basis with respect to the lexicographic term order determined by $w>x>y>z$.  Then $N_{w,I} = (yz-x^2)$, which is a complete intersection, and $C_{w,I} = (yz-x^2, z^2, xz)$, which is $(x,z)$-primary (hence unmixed).  By \cite[Corollary 4.3]{KR21}, $I$ is G-linked to $C_{w,I}$ by elementary G-biliaison.
    
    The given generators of $C_{w,I}$ form a Gr\"obner basis with respect to any $x$-compatible term order.  (Note that, although $C_{w,I}$ is linear in $y$, the geometric vertex decomposition with respect to $y$ is degenerate.)  The polarization of the $x$-compatible Gr\"obner basis of $C_{w,I}$ is $\{z^2, xz, xx'-zy\}$, which is again a Gr\"obner basis.  (Equivalently, by \cref{thm:inducedGB}, $x-x'$ is not a zerodivisor on the quotient by $(z^2, xz, xx'-zy)$.)  In the notation of \cref{thm:polarLink} (with $C_{w,I}$ now playing the role of $I$), we have $N = (z^2)$, which is a complete intersection ideal, hence defines a Cohen--Macaulay and $G_0$ quotient, and $D = (x,z)$.  Thus, by \cref{thm:polarLink}, $C_{w,I}$ is G-linked to $D$ by elementary G-biliaison.  Hence, $I$ is G-linked to the complete intersection ideal $D$ by a sequence of elementary G-biliaisons.
        \end{example}

    We note that, in \cref{ex:mainTheroemApp1}, one cannot immediately polarize the given Gr\"obner basis of $I$ at any of the variables and obtain a Gr\"obner basis of the polarization under the induced term order.  This example is representative of our typical experience in using \cref{thm:polarLink}.  It occurs often that the polarization of a Gr\"obner basis fails to be a Gr\"obner basis. However, it was enough that the simpler ideal $C_{w,I}$, which is G-linked to $I$, have this property.

    \begin{example}
        For $t \geq 2$, if $I$ is an ideal with a generating set of the form $\{y^t q_1+r_1, \ldots, y^t q_k+r_k, h_1, \ldots, h_\ell\}$, where $y$ does not divide any $q_i$, $r_i$, or $h_i$, then the polarization of the reduced Gr\"obner basis of any $y$-compatible term order is again a Gr\"obner basis by \cref{prop:uniform-deg-in-y-polarizes}.  We may then apply \cref{thm:polarLink} $t$ times. 

        An alternate pathway to studying this class of ideals in the context of Gorenstein liaison is discussed in detail in \cite{CD}.

        Returning to \cref{ex:uniform-y-degree-polarizes}, in which we considered the Gr\"obner basis \[
        \G = \{y^2a+yd^2, y^2c+b^3, cd^2,b^3d^2,ab^3\}
        \] with respect to the lexicographic term order determined by $y>a>b>c>d$, which generates an unmixed ideal of height $3$. We saw that  
    \[
   \P_y(\G) =  \{y'ya+yd^2, y'yc+b^3, cd^2, b^3d^2, ab^3 \}
    \] is  a Gr\"obner basis under the induced term order of \cref{not:indOrder}. Via \cref{thm:polarLink}, the geometric vertex decomposition \[
 \init_y(\P_y(\G)) = (y'a+d^2, y'c, cd^2, b^3d^2, ab^3) \cap (cd^2, b^3d^2, ab^3, y) 
    \] induces the elementary G-biliaison \[
    (\G)/(cd^2, b^3d^2, ab^3) \cong [(ya+d^2, yc, cd^2, b^3d^2, ab^3)/(cd^2, b^3d^2, ab^3)](-1).
    \] Thus $(\G)$ is G-linked to $(ya+d^2, yc, cd^2, b^3d^2, ab^3)$ in two steps. Then, using a geometric vertex decomposition at any of $y$, $a$, or $d$, one readily checks that $(ya+d^2, yc, cd^2, b^3d^2, ab^3)$ and, hence, $(\G)$ are glicci.
    \end{example}

\section{Acknowledgements}
 This project began as part of the Women in Commutative Algebra and Algebraic Geometry (WICAAG) virtual workshop, organized by Megumi Harada and Claudia Miller and hosted by the Fields Institute.  The authors thank Fields Institute, the other participants, and the especially the organizers of WICAAG for bringing us together and providing a supportive community. 

 We are very grateful to Elisa Gorla for comments on an earlier version of this document. We thank also Craig Huneke for comments concerning surrounding literature.  Computations performed with Macaulay2 \cite{M2} were very helpful during our project. Finally, we thank the anonymous referee for valuable feedback and especially for asking about the preservation of Betti numbers under geometric polarization, which led to a major improvement in the statement and proof of \cref{prop:stFacts}.


\bibliographystyle{amsalpha}
\bibliography{refs}

\end{document}